\documentclass[sn-mathphys,Numbered]{sn-jnl}

\usepackage{graphicx}%
\usepackage{multirow}%
\usepackage{amsmath,amssymb,amsfonts}%
\usepackage{amsthm}%
\usepackage{mathrsfs}%
\usepackage[title]{appendix}%
\usepackage{xcolor}%
\usepackage{textcomp}%
\usepackage{manyfoot}%
\usepackage{booktabs}%
\usepackage{algorithm}%
\usepackage{algorithmicx}%
\usepackage{algpseudocode}%
\usepackage{listings}%

\theoremstyle{thmstyleone}%
\newtheorem{theorem}{Theorem}
\newtheorem{corollary}{Corollary}
\newtheorem{lemma}{Lemma}
\newtheorem{proposition}{Proposition}%

\theoremstyle{thmstyletwo}%
\newtheorem{example}{Example}%
\newtheorem{remark}{Remark}%

\theoremstyle{thmstylethree}%

\begin{document}
	
	\title[Multiwindow Gabor systems]{Multi-window Gabor systems on discrete periodic sets}
	
	
	\author*[]{\fnm{Najib} \sur{Khachiaa}}\email{khachiaa.najib@uit.ac.ma}

	\affil[]{\orgdiv{Laboratory Partial Differential Equations, Spectral Algebra and Geometry, Department of Mathematics}, \orgname{Faculty of Sciences, University Ibn Tofail}, \orgaddress{\city{Kenitra}, \country{Morocco}}}

		\abstract{In this paper, we study multiwindow discrete Gabor $(M-D-G)$ systems \\$\mathcal{G}(g,L,M,N)$ on discrete periodic sets $\mathbb{S}$ and give some necessary and/or sufficient matrix-conditions for a $M-D-G$ system in $\ell^2(\mathbb{S})$ to be a frame. We characterize, also, which $M-D-G$ frames for $\ell^2(\mathbb{S})$ are Riesz bases by the parameters $L$, $M$ and $N$. Matrix-characterizations of  $M-D-G$ Parseval frames and $M-D-G$ orthonormal bases for $\ell^2(\mathbb{S})$  are  also given. Then, in the case of $\mathbb{S}=\mathbb{Z}$, we characterize  the existence of $M-D-G$ frames, $M-D-G$ Parseval frames, $M-D-G$ Riesz bases and $M-D-G$ orthonormal bases for $\ell^2(\mathbb{Z})$ by the parameters $M$, $N$ and $L$. We present, also, a matrix-characterization of dual $M-D-G$ frames for $\ell^2(\mathbb{S})$. A perturbation matrix-condition of $M-D-G$ frames for $\ell^2(\mathbb{S})$ is also prsented. We, then, show that a pair of $M-D-G$ Bessel systems can generate pairs of  M-D-G dual frames in the case of $\mathbb{S}=\mathbb{Z}$. By the Zak-transform, characterizations  of complete M-D-G systems and  M-D-G frames for $\ell^2(\mathbb{Z})$  are given in the case of $M=N$ and necessary conditions for a M-D-G system to be a Riesz basis/ orthonormal basis for $\ell^2(\mathbb{Z})$ are also given. We, also, study M-D-G  $K$-frames in $\ell^2(\mathbb{S})$, where $K\in B(\ell^2(\mathbb{S})\,)$, and presente some sufficient matrix-conditions for a M-D-G system to form a $K$-frame and give a construction method of M-D-G $K$-frames which are not M-D-G frames and  some examples.}
	

	
	\maketitle
	\section{Introduction and preliminaries}
	When a signal appears periodically but intermittently, it can be considered within the entire space $\ell^2(\mathbb{Z})$ and analyzed in the standard manner. However, if the signal is only emitted for short periods, this method might not be the best approach. To perform Gabor analysis of the signal most efficiently while preserving all its features, Li and Lian studied single window Gabor systems on discrete periodic sets. They derived density results and frame characterizations. Compared to single window Gabor systems, multiwindow Gabor systems can be both interesting and beneficial, as they allow for more flexibility by using windows of different types and widths. Q-F. Lian, J. Gong, and M-H. You explored these multiwindow Gabor systems, achieving frame characterizations and notable results. 
	
	A sequence $\{f_i\}_{i\in \mathcal{I}}$, where $\mathcal{I}$ is a countable set, in a separable Hilbert space $H$ is said to be frame if there exist $0< A\leq B<\infty$ ( called frame bounds) such that for all $f\in H$, $$A\|f\|^2\leq \displaystyle{\sum_{i\in \mathcal{I}}\vert \langle f,f_i \rangle\vert^2}\leq B\|f\|^2.$$ If only the upper inequality holds,  $\{f_i\}_{i\in \mathcal{I}}$ is called  a Bessel sequence with Bessel bound $B$. If $A=B$, the sequence is called a tight frame and if $A=B=1$, it is called a Parseval frame for $H$. If two frames $\{f_i\}_{i\in \mathcal{I}}$ and $\{g_i\}_{i\in \mathcal{I}}$ in $H$ satisfy for all $f\in H$, 
	$$f=\displaystyle{\sum_{i\in \mathcal{I}}\langle f, g_i\rangle f_i}.$$
We say that $\{g_i\}_{i\in \mathcal{I}}$ is a dual frame of $\{f_i\}_{i\in \mathcal{I}}$. It is well known that the dual frames can exchange with each other. Then we can, simply, say that $\{g_i\}_{i\in \mathcal{I}}$ and $\{f_i\}_{i\in \mathcal{I}}$ are dual frames in $H$. If $\{f_i\}_{i\in \mathcal{I}}$ is a Bessel sequence in $H$, then the pre-frame (synthesis) operator $U$ defined as follows,
$$\begin{array}{rcl}
U:\ell^2(\mathcal{I})&\longrightarrow& H\\
\{a_i\}_{i \in \mathcal{I}}&\mapsto& \displaystyle{\sum_{i\in \mathcal{I}}a_i f_{i}},
\end{array}$$
is well defined, linear and bounded operator. Denote by $\theta$ the adjoint of $U$, $\theta$ is called the analysis (transform) operator and is defined, explicitely, for all $f\in H$ by $\theta(f):=\{\langle f,f_i\rangle\}_{i\in \mathcal{I}}$. The frame operator, often denoted by $S$, is the composite $U\theta$ and is defined, explicitely, for all $f\in H$ by $S(f)=\displaystyle{\sum_{i\in \mathcal{I}}\langle f, f_i\rangle f_i}.$ If $\{f_i\}_{i\in \mathcal{I}}$ is a frame, then $S$ is a self-adjoint, positive and invertible operator. In this case, $\{f_i\}_{i\in \mathcal{I}}$ and $\{S^{-1}f_i\}_{i\in \mathcal{I}}$ are dual frames. For more details on frame theory, the reader can refer to \cite{1}.

Denote by $\mathbb{N}$ the set of positive integers, i.e. $\mathbb{N}:=\{1,2,3,...\}$ and for a given $K\in \mathbb{N}$, write $\mathbb{N}_{K}:=\{0,1,...,K-1\}$. Let $N,M,L\in \mathbb{N}$. A nonempty subset $\mathbb{S}$ of $\mathbb{Z}$ is said to be $N\mathbb{Z}$-periodic set if for all $j\in \mathbb{S}$ and for all $n\in \mathbb{Z}$, $j+nN \in \mathbb{S}$. For $K\in \mathbb{N}$, write $\mathbb{S}_K:=\mathbb{S}\cap \mathbb{N}_K$. We denote by $\ell^2(\mathbb{S})$ the closed subspace of $\ell^2(\mathbb{Z})$ defined by, $$\ell^2(\mathbb{S}):=\{f\in \ell^2(\mathbb{Z}):\, f(j)=0 \text{ if } j\notin  \mathbb{S} \}.$$ 
And by $\ell_0(\mathbb{S})$ the dense subspace in $\ell^2(\mathbb{S})$ defined by,  $$\ell_0(\mathbb{S}):=\{f\in \ell^2(\mathbb{S}):\; supp(f) \text{ is finite}\}.$$
Define the modulation operator $E_{\frac{m}{M}}$ with $m\in \mathbb{Z}$ and the translation operator $T_{nN}$ with $n\in \mathbb{Z}$ for $f\in \ell^2(\mathbb{S})$ by $$E_{\frac{m}{M}}f(.):=e^{2\pi i \frac{m}{M}.} f(.), \;\; T_{nN}f(.):=f(.-nN).$$ 
The modulation and translation  operators are unitary operators of $\ell^2(\mathbb{S})$.\\ For $g:=\{g_l\}_{l\in \mathbb{N}_L} \subset \ell^2(\mathbb{S})$, the associated multiwindow discrete Gabor system (M-D-G) is given by,
$$\mathcal{G}(g,L,M,N):=\{E_{\frac{m}{M}}T_{nN}g_l\}_{m\in \mathbb{N}_M,n\in \mathbb{Z}, l\in \mathbb{N}_L}.$$ 
For $j\in \mathbb{Z}$, we denote by $B_j:=\{k\in \mathbb{Z}:\, j+kM\in \mathbb{S}\}$ and by, \\$J(j):=diag(...,\chi_{B_j}(-1),\chi_{B_j}(0),\chi_{B_j}(1),...)$. We associate $\phi \in \ell^2(\mathbb{S})$ with a bi-infinite matrix-valued $\mathcal{M}_{\phi}$ defined for each $j\in \mathbb{Z}$ by $\mathcal{M}_{\phi}(j):=\{\phi(j+kM-nN)\}_{k,n\in \mathbb{Z}}$. In what follows, $\mathcal{M}_{\phi}^*(.)$ is the conjugate transpose of $\mathcal{M}_{\phi}(.)$. Q-F. Lian, J. Gong and M-H. You studied  M-D-G frames and presented characterizations of M-D-G frames $\mathcal{G}(g,L,M,N)$  by $J(.)$ and $\mathcal{M}_{g_l}(.)$. 

\begin{theorem}\label{thm1} \cite{8}
If $\mathcal{G}(g,L,M,N)$ is M-D-G frame for $\ell^2(\mathbb{S})$ with frame bounds $A$ and $B$, then for $j\in \mathbb{S}$,
$$\displaystyle{\frac{A}{M}}\leq \displaystyle{\sum_{l\in \mathbb{N}_L} \left( \mathcal{M}_{g_l}(j)\mathcal{M}_{g_l}^*(j) \right)_{0,0}}\leq \displaystyle{\frac{B}{M}}.$$
\end{theorem}

\begin{theorem}\label{thm2}\cite{8}
If $B:=M \displaystyle{\max_{j\in \mathbb{S}\cap \mathbb{N}_N} \sum_{l\in \mathbb{N}_L}\sum_{k\in \mathbb{Z}}\left\vert \left(\mathcal{M}_{g_l}(j)\mathcal{M}_{g_l}(j) \right)_{0,k}\right\vert}< \infty$, \\
then $\mathcal{G}(g,L,M,N)$ is a Bessel sequence with bessel bound $B$.\\
If also $A:=M\displaystyle{\min_{j\in \mathbb{S}\cap \mathbb{N}_N }\sum_{l\in \mathbb{N}_N} \left[ \mathcal{M}_{g_l}(j)\mathcal{M}_{g_l}^*(j)_{0,0}-\sum_{k\in \mathbb{Z}-\{0\}}\left \vert \left(\mathcal{M}_{g_l}(j)\mathcal{M}_{g_l}^*(j)\right)_{0,k}\right\vert \right]}> 0$, \\
then  $\mathcal{G}(g,L,M,N)$ is a frame for $\ell^2(\mathbb{S})$ with frame bounds $A$ and $B$.
\end{theorem}
As a direct result of the above theorem, if $g\in \ell_0(\mathbb{S})$, then $\mathcal{G}(g,L,M,N)$ is a Bessel sequence in $\ell^2(\mathbb{S})$.

The discrete Zak tansform $z_M$ of $f\in \ell^2(\mathbb{Z})$ for $j\in \mathbb{Z}$ and $\theta\in \mathbb{R}$ is defined by, 	
$$z_Mf(j,\theta):=\displaystyle{\sum_{k\in \mathbb{Z}}f(j+kM)e^{2\pi i k \theta}}.$$
It is well known that $z_M$ is, by quasiperiodicity, completely defined by its values for $j\in \mathbb{N}_M$ and $\theta \in [0,1[$ and that  $z_M$ is a unitary linear operator from $\ell^2(\mathbb{Z})$ to the Hilbert space  $\ell^2(q)$ where $q=\mathbb{N}_M\times [0,1[$ and 
$$\ell^2(q):=\{\psi:q \rightarrow \mathbb{C}:\, \displaystyle{\sum_{j=0}^{M-1}\int_{0}^1 \vert \psi(j,\theta)\vert^2 \, d\theta}< \infty\}.$$ The reader can refer to \cite{4} for more details on discrete Zak transform.

As a generalization of frames, $K$-frames are also widely studied in abstract Hilbert spaces. Giving a separable Hilbert space $H$ and  an operator $K\in B(H)$. A sequence $\{f_i\}_{i\in \mathcal{I}}$ is said to be a $K$-frame for $H$ if there exist $0<A\leq B<\infty$ such that for all $f\in H$,
$$A\|K^*(f)\|^2\leq \displaystyle{\sum_{i\in \mathcal{I}}\vert \langle f, f_i\rangle \vert^2}\leq B\|f\|^2.$$
If  $\forall f\in H$, $A\|K^*(f)\|^2=\displaystyle{\sum_{i\in \mathcal{I}}\vert \langle f, f_i\rangle \vert^2}$, we say that $\{f_i\}_{i\in \mathcal{I}}$ is a $K$-tight frame and, moreover, if $A=1$ it is called a Parseval $K$-frame. Note that when $K=I_H$, the operator identity in $H$, a $K$-frame is simply an ordinary frame. While a $K$-frame is in particular a Bessel sequence, then its synthesis operator is well defined and bounded. A $K$-frame is said to be $K$-minimal frame if its synthesis operator is injective. For more details on $K$-frames, the reader can refer to  \cite{3} and \cite{11}.

In this paper, we study the M-D-G frames and M-D-G $K$-frames for $\ell^2(\mathbb{S})$. The work is organized as follows. In section 2, we will present some auxiliary lemmas to be used in the following sections. In section 3, we give some sufficient and/or necessary matrix-conditions for a M-D-G system to be frame in $\ell^2(\mathbb{S})$. We characterize, also, which M-D-G frames are Riesz bases by the parameters $L$, $M$ and $N$. Matrix-characterizations of Parseval M-D-G frames and M-D-G orthonormal bases  are  also given. Then, we characterize  the existence of M-D-G frames, M-D-G Parseval frames, M-D-G Riesz bases and M-D-G orthonormal bases for $\ell^2(\mathbb{S})$ by the parameters $M$, $N$ and $L$. A matrix-characterization of dual M-D-G frames in $\ell^2(\mathbb{S})$ is also presented. We give a perturbation matrix-condition of M-D-G frames and finish the section by showing that every pair of M-D-G Bessel systems can generate  pairs  of M-D-G dual frames for $\ell^2(\mathbb{S})$. In section 4, by the zak transform,  we characterize complete M-D-G systems and M-D-G frames  for $\ell^2(\mathbb{Z})$ in the case of $M=N$ and  give some necessary conditions for a M-D-G system in $\ell^2(\mathbb{Z})$ to be a Riesz basis/orthonormal basis for $\ell^2(\mathbb{Z})$, and give some examples.  In section 5, we study M-D-G $K$-frames for $\ell^2(\mathbb{S})$ and give some examples of such frames which are not ordinary frames.
	
	\section{Some auxiliary lemmas}
	In this section, we provide some lemmas and notations to be used in the following sections. In addition to all the notations introduced in the introduction, we add the notation $\vert A\vert=\max{A}-\min{A}$ , where $A\subset \mathbb{Z}$. Then, we provide some useful lemmas for the rest.
	\begin{lemma}\label{lem1}
	For $j\in \mathbb{Z}$, we have,
	$$\displaystyle{\sum_{m=0}^{M-1}e^{2\pi i j \frac{m}{M}}}=\left \lbrace
	\begin{array}{rcl}
	M, &\text{ if } j\in M\mathbb{Z}&\\
	0, &\text{ otherwise}.
	\end{array}
	\right.$$
	\end{lemma}
	\begin{lemma}\label{lem1'}
	Denote $\ell^2(\mathbb{N}_M)$ the space of all $M$-periodic sequences.\\ The sequence $\{\displaystyle{\frac{1}{\sqrt{M}}}e^{2\pi i\frac{m}{M}.}\}_{m\in \mathbb{N}_M}$ is an orthonormal basis for $\ell^2(\mathbb{N}_M)$.
	\end{lemma}
	
	\begin{lemma}\label{lem2}
	Assume that $\mathcal{G}(g,L,M,N)$ and $\mathcal{G}(h,L,M,N)$ are two Bessel M-D-G systems in $\ell^2(\mathbb{S})$. Let $S_{h,g}:=U_g\theta_h$ where $U_g$ and $\theta_h$ are, respectively, the synthesis operator of $\mathcal{G}(g,L,M,N)$ and the analysis operator of $\mathcal{G}(h,L,M,N)$. Then for all $m\in \mathbb{N}_M$ and $n\in \mathbb{Z}$,
	$$E_{\frac{m}{M}}T_{nN}S_{h,g}=S_{h,g}E_{\frac{m}{M}}T_{nN}.$$
Moreover, if $S_{h,g}$ is invertible, then  for all $m\in \mathbb{N}_M$ and $n\in \mathbb{Z}$,
$$E_{\frac{m}{M}}T_{nN}S^{-1}_{h,g}=S^{-1}_{h,g}E_{\frac{m}{M}}T_{nN}.$$
	\end{lemma}

	\begin{lemma}\label{lem3}\cite{1}
	Assume that $\{f_i\}_{i\in \mathcal{I}}$ is a frame for a seperable Hilbert space $H$ with frame bounds $A$ and $B$. Let $\{g_i\}_{i\in \mathcal{I}}\subset H$. Suppose that there exists $0< R < A$ such that for all $f\in H$, 
	$$\displaystyle{\sum_{i\in \mathcal{I}}\left \vert \langle f, fi-g_i\rangle \right \vert ^2}\leq R\|f\|^2,$$
	then $\{g_i\}_{i\in \mathcal{I}}$ is a frame for $H$ with bounds 
	$$A\left(1-\displaystyle{\sqrt{\frac{R}{A}}}\right)^2, \; \;  B\left(1+\displaystyle{\sqrt{\frac{R}{B}}}\right)^2.$$
	\end{lemma}
	
	\begin{lemma}\label{lem3'}\cite{1}
Let $\{f_i\}_{i\in \mathcal{I}}$ be a Parseval frame for a separable Hilbert space $H$. Then the following statements are equivalent:
\begin{enumerate}
\item $\{f_i\}_{i\in \mathcal{I}}$ is a Riesz basis.
\item $\{f_i\}_{i\in \mathcal{I}}$ is an orthonormal basis.
\item For all $i\in \mathcal{I}$, $\|f_i\|=1.$
\end{enumerate}
\end{lemma}	
	The  next result  is very useful for the next.
	\begin{lemma}\label{lem4}
	Let $\mathcal{G}(g,L,M,N)$ be a M-D-G system in $\ell^2(\mathbb{S})$ as defined before. Then for all $f\in \ell_0(\mathbb{S})$, we have:
	$$\begin{array}{rcl}
	&&\displaystyle{\sum_{l\in \mathbb{N}_L}\sum_{n\in \mathbb{Z}}\sum_{m\in \mathbb{N}_M}\left \vert \langle f, E_{\frac{m}{M}}T_{nN}g_l\rangle \right\vert^2}=
	M\displaystyle{\sum_{j\in \mathbb{S}}\left(\sum_{l\in \mathbb{N}_L}\mathcal{M}_{g_l}(j)\mathcal{M}_{gl}^*(j)\right)_{0,0}\vert f(j)\vert^2}\\
	&+& M\displaystyle{\sum_{k\in \mathbb{Z}-\{0\}}\sum_{j\in \mathbb{S}}\left(\sum_{l\in \mathbb{N}_L}\mathcal{M}_{g_l}(j)\mathcal{M}_{gl}^*(j)\right)_{0,k}f(j)\overline{f(j+kM)}}.
	\end{array}$$
	\end{lemma}
	\begin{proof}
	Let $f\in \ell_0(\mathbb{S})$. We have: $$\begin{array}{rcl}
	&&\displaystyle{\sum_{l\in \mathbb{N}_L}\sum_{n\in \mathbb{Z}}\sum_{m\in \mathbb{N}_M}\left \vert \langle f, E_{\frac{m}{M}}T_{nN}g_l\rangle \right\vert^2}=\displaystyle{\sum_{l\in \mathbb{N}_L}\sum_{n\in \mathbb{Z}}\sum_{m\in \mathbb{N}_M}\left \vert \sum_{j\in \mathbb{Z}}f(j)\overline{g_l(j-nN)}e^{-2\pi i \frac{m}{M}j}\right \vert^2}\\
	&=& \displaystyle{\sum_{l\in \mathbb{N}_L}\sum_{n\in \mathbb{Z}}\sum_{m\in \mathbb{N}_M} \left(\displaystyle{\sum_{j\in \mathbb{S}}f(j)\overline{g_l(j-nN)}e^{-2\pi i \frac{m}{M}j}}\right).\left(\displaystyle{\sum_{p\in \mathbb{S}}\overline{f(p)}g_l(p-nN)e^{2\pi i \frac{m}{M}p}}\right)}\\
	&=&\displaystyle{\sum_{l\in \mathbb{N}_L}\sum_{n\in \mathbb{Z}}\sum_{m\in \mathbb{N}_M} \sum_{j\in \mathbb{S}} \vert f(j)\vert^2 \, \vert g_l(j-nN)\vert^2}\\&+&\displaystyle{\sum_{l\in \mathbb{N}_L}\sum_{n\in \mathbb{Z}}\sum_{m\in \mathbb{N}_M} \sum_{j,p\in \mathbb{S}, j\neq p} f(j)\overline{g_l(j-nN)} \,\overline{f(p)}\, g_l(p-nN) e^{2\pi i\frac{m}{M}(p-j)}}.
	\end{array}$$
	Since $f\in \ell_0(\mathbb{S})$, then we can change the summation order. And using lemma \ref{lem1}, we have:
	$$\begin{array}{rcl}
	&&\displaystyle{\sum_{l\in \mathbb{N}_L}\sum_{n\in \mathbb{Z}}\sum_{m\in \mathbb{N}_M}\left \vert \langle f, E_{\frac{m}{M}}T_{nN}g_l\rangle \right\vert^2}= M\displaystyle{\sum_{l\in \mathbb{N}_L}\sum_{j\in \mathbb{S}}\, \left(\sum_{n\in \mathbb{Z}} \vert g_l(j-nN)\vert^2\right)\,\vert f(j)\vert^2 }\\
	&+& \displaystyle{\sum_{l\in \mathbb{N}_L}\sum_{j,p\in \mathbb{S},\,j\neq p}\, \left(\sum_{n\in \mathbb{Z}} \overline{g_l(j-nN)} \, g_l(p-nN)\;\right) \, f(j)\, \overline{f(p)} \sum_{m\in \mathbb{N}_M} e^{2\pi i\frac{m}{M}(p-j)}}\\

&=& M\displaystyle{\sum_{l\in \mathbb{N}_L}\sum_{j\in \mathbb{S}} \left(\mathcal{M}_{g_l}(j)\mathcal{M}_{g_l}^*(j)\right)_{0,0}\,\vert f(j)\vert^2 \, }\\
	&+& M\displaystyle{\sum_{l\in \mathbb{N}_L}\sum_{j\in \mathbb{S}}\; \sum_{p\in j+M(\mathbb{Z}-\{0\})} 
	 \left(\sum_{n\in \mathbb{Z}} \overline{g_l(j-nN)} \, g_l(p-nN)\;\right)\,f(j)\overline{f(p)}}\\
	&=& M\displaystyle{\sum_{l\in \mathbb{N}_L}\sum_{j\in \mathbb{S}} \left(\mathcal{M}_{g_l}(j)\mathcal{M}_{g_l}^*(j)\right)_{0,0}\,\vert f(j)\vert^2 \, }\\
	&+& M\displaystyle{\sum_{l\in \mathbb{N}_L}\sum_{j\in \mathbb{S}}\sum_{k\in \mathbb{Z}-\{0\}} \left(\mathcal{M}_{g_l}(j)\mathcal{M}_{g_l}^*(j)\right)_{0,k} f(j)\overline{f(j+kM)}}.
	\end{array}$$
	The fact that $f\in \ell_0(\mathbb{S})$ completes the proof since we can change the summation order.
	\end{proof}
	\begin{remark}
	For the rest of this paper, we write for all $f\in \ell_0(\mathbb{S})$, $$\displaystyle{\sum_{l\in \mathbb{N}_L}\sum_{n\in \mathbb{Z}}\sum_{m\in \mathbb{N}_M}\left \vert \langle f, E_{\frac{m}{M}}T_{nN}g_l\rangle \right\vert^2}=F_1(f)+F_2(f).$$
	Where $F_1(f):=M\displaystyle{\sum_{j\in \mathbb{S}}\left(\sum_{l\in \mathbb{N}_L}\mathcal{M}_{g_l}(j)\mathcal{M}_{gl}^*(j)\right)_{0,0}\vert f(j)\vert^2}$ and\\ $F_2(f):=M\displaystyle{\sum_{k\in \mathbb{Z}- \{0\}}\sum_{j\in \mathbb{S}}\left(\sum_{l\in \mathbb{N}_L}\mathcal{M}_{g_l}(j)\mathcal{M}_{gl}^*(j)\right)_{0,k}f(j)\overline{f(j+kM)}}$.
	\end{remark}
	
	\begin{lemma}\label{lem4'}
	Let $\mathcal{G}(g,L,M,N)$ and $\mathcal{G}(h,L,M,N)$ be two M-D-G systems in $\ell^2(\mathbb{S})$. For all $f,\phi\in \ell_0(\mathbb{S})$, we have:
	$$
	\displaystyle{\sum_{l\in \mathbb{N}_L}\sum_{n\in \mathbb{Z}}\sum_{m\in \mathbb{N}_M}\langle f,E_{\frac{m}{M}}T_{nN}g_l\rangle\,\langle E_{\frac{m}{M}}T_{nN}h_l,\phi\rangle}
	=M\displaystyle{\sum_{j\in \mathbb{S}}\sum_{k\in \mathbb{Z}}\left(\sum_{l\in \mathbb{N}_L}\mathcal{M}_{h_l}(j)\mathcal{M}_{g_l}^*(j)\right)_{0,k}\,f(j+kM)\overline{\phi(j)}}.$$
	\end{lemma}
\begin{proof}
	Let $f,\phi\in \ell_0(\mathbb{S})$. We have: 
	$$\begin{array}{rcl}
	&&\displaystyle{\sum_{l\in \mathbb{N}_L}\sum_{n\in \mathbb{Z}}\sum_{m\in \mathbb{N}_M}\langle f,E_{\frac{m}{M}}T_{nN}g_l\rangle\,\langle E_{\frac{m}{M}}T_{nN}h_l,\phi\rangle}\\
	&=&\displaystyle{\sum_{l\in \mathbb{N}_L}\sum_{n\in \mathbb{Z}}\sum_{m\in \mathbb{N}_M}\left(\sum_{j\in \mathbb{N}_M}\sum_{k\in \mathbb{Z}}f(j+kM)\overline{g_l(j+kM-nN)}e^{-2\pi i\frac{m}{M}j}\right)}\\
	&\,\times\,&\displaystyle{\left(\sum_{j\in \mathbb{N}_M}\sum_{k\in \mathbb{Z}}e^{2\pi i\frac{m}{M}j}h_l(j+kM-nN)\overline{\phi(j+kM)}\right) }\\
	&=& \displaystyle{\sum_{l\in \mathbb{N}_L}\sum_{n\in \mathbb{Z}}\sum_{m\in \mathbb{N}_M}\left(\sum_{j\in \mathbb{N}_M}F_n(j)e^{-2\pi i\frac{m}{M}j}\right)}
	\,\times\,\displaystyle{\left(\sum_{j\in \mathbb{N}_M}e^{2\pi i\frac{m}{M}j}\,\overline{H_n(j)}\right) }\\
	&=& \displaystyle{\sum_{l\in \mathbb{N}_L}\sum_{n\in \mathbb{Z}}\sum_{m\in \mathbb{N}_M}\, \langle F_n,e^{-2\pi i\frac{m}{M}.}\rangle}
	\,\displaystyle{\langle e^{2\pi i\frac{m}{M}.},H_n\rangle}\\
    &=& M\displaystyle{\sum_{l\in \mathbb{N}_L}\sum_{n\in \mathbb{Z}}\, \langle F_n,H_n\rangle} \;\;\;\;(\text{ lemma }\ref{lem1'})\\
  &=&M\displaystyle{\sum_{l\in \mathbb{N}_L}\sum_{n\in \mathbb{Z}}\sum_{j\in \mathbb{N}_M} F_n(j)\overline{H_n(j)}}\\
&=& M\displaystyle{\sum_{l\in \mathbb{N}_L}\sum_{n\in \mathbb{Z}}\sum_{j\in \mathbb{N}_M} F_n(j)\, \sum_{k\in \mathbb{Z}} h_l(j+kM-nN)\overline{\phi(j+kM)}}\\
&=& M\displaystyle{\sum_{l\in \mathbb{N}_L}\sum_{n\in \mathbb{Z}}\sum_{j\in \mathbb{N}_M}\sum_{k\in \mathbb{Z}} F_n(j+kM)  h_l(j+kM-nN)\overline{\phi(j+kM)}}\; \;( F_n \text{ is M-periodic})\\
&=& M\displaystyle{\sum_{l\in \mathbb{N}_L}\sum_{n\in \mathbb{Z}}\sum_{j\in \mathbb{Z}}F_n(j)  h_l(j-nN)\overline{\phi(j)}}\\
&=& M\displaystyle{\sum_{l\in \mathbb{N}_L}\sum_{n\in \mathbb{Z}}\sum_{j\in \mathbb{S}}F_n(j)  h_l(j-nN)\overline{\phi(j)}}\;\;\; \text{ since }\phi\in \ell_0(\mathbb{S})\\
\end{array}$$
\newpage
$$\begin{array}{rcl}
&=&M\displaystyle{\sum_{l\in \mathbb{N}_L}\sum_{n\in \mathbb{Z}}\sum_{j\in \mathbb{Z}} \sum_{k\in \mathbb{Z}}f(j+kM)\overline{g_l(j+kM-nN)}\,   h_l(j-nN)\overline{\phi(j)}}\\
&=&M\displaystyle{\sum_{l\in \mathbb{N}_L}\sum_{j\in \mathbb{Z}} \sum_{k\in \mathbb{Z}}\left(\sum_{n\in \mathbb{Z}}h_l(j-nN)\overline{g_l(j+kM-nN)}\right) f(j+kM)\, \overline{\phi(j)}}\\
&=&M\displaystyle{\sum_{l\in \mathbb{N}_L}\sum_{j\in \mathbb{Z}} \sum_{k\in \mathbb{Z}}\left(\mathcal{M}_{h_l}(j)\mathcal{M}_{g_l}^*(j)\right)_{0,k} f(j+kM)\, \overline{\phi(j)}}\\
&=&M\displaystyle{\sum_{j\in \mathbb{Z}} \sum_{k\in \mathbb{Z}}\left(\sum_{l\in \mathbb{N}_L}\mathcal{M}_{h_l}(j)\mathcal{M}_{g_l}^*(j)\right)_{0,k} f(j+kM)\, \overline{\phi(j)}}.\\
	\end{array}$$
Where $$F_n(.):=\displaystyle{\sum_{k\in \mathbb{Z}}f(.+kM)\,\overline{g_l(.+kM-nN)}},\;H_n(.):=\displaystyle{\sum_{k\in \mathbb{Z}}\phi(.+kM)\,\overline{h_l(.+kM-nN)}}$$ and we could change the summation order since $f,\phi\in \ell_0(\mathbb{S})$.
	\end{proof}	
	
	\begin{lemma}\label{lem5}
	For all $\phi,\psi \in \ell^2(\mathbb{S})$, $\mathcal{M}_{\phi}(.)\mathcal{M}_{\psi}^*(.)$ is $N$-periodic. i.e: For all $j\in \mathbb{S}$ and $n\in \mathbb{Z}$, $\mathcal{M}_{\phi}(j)\mathcal{M}_{\phi}^*(j)=\mathcal{M}_{\phi}(j+nN)\mathcal{M}_{\psi}^*(j+nN)$.
	\end{lemma}
	\begin{proof}
	Let $j\in \mathbb{S}$,  $n,k,k'\in \mathbb{Z}$, we have:
	$$\begin{array}{rcl}
	\left(\mathcal{M}_{\phi}(j+nN)\mathcal{M}_{\psi}^*(j+nN)\right)_{k,k'}&=&\displaystyle{\sum_{p\in \mathbb{Z}}\phi(j+nN+kM-pN)\overline{\psi(j+nN+k'M-pN)}}\\
	&=&\displaystyle{\sum_{p\in \mathbb{Z}}\phi(j+kM+(n-p)N)\overline{\psi(j+k'M+(n-p)N)}}\\
	&=& \displaystyle{\sum_{q\in \mathbb{Z}}\phi(j+kM-qN)\overline{\psi(j+k'M-qN)}}\;\;\; (q=n-p)\\
	&=&\left(\mathcal{M}_{\phi}(j)\mathcal{M}_{\psi}^*(j)\right)_{k,k'}.
	\end{array}$$
	\end{proof}
	\begin{lemma}\label{lem6}\cite{4}
	Let $M\in \mathbb{N}$. Assume that $g\in \ell^2(\mathbb{Z})$ such that $z_Mg$ is continuous. 
	\begin{enumerate}
	\item If $g$ is odd, then $z_Mg(0,0)=z_Mg(0,\displaystyle{\frac{1}{2}})=0.$ Moreover, if $M$ is even, then $z_Mg(\displaystyle{\frac{M}{2}},0)=0.$
	\item If $g$ is even and $M$ is even, then $z_Mg(\displaystyle{\frac{M}{2}},\displaystyle{\frac{M}{2}})=0$.
	\end{enumerate}
	\end{lemma}
	
	\begin{lemma}\label{lem7}\cite{4}
	Let $M\in \mathbb{N}$.\\
If $g\in \ell^1(\mathbb{Z})$, then for all 
$j\in \mathbb{Z}$, $z_Mg(j,.)$ is continuous on $\mathbb{R}$.
		\end{lemma}
	We finish this section with a famous theorem, called Douglas theorem, which is very useful in $K$-frame theory.
	\begin{theorem}\label{thm3}\cite{2}
	Let $U,V\in B(H)$ where $H$ is a separable Hilbert space. The following statements are equivalent: 
	\begin{enumerate}
	\item $R(U)\subset R(V)$.
	\item There exists $c> 0$ such that $UU^* \leq c VV^*.$
	\item There exists $X\in B(H)$ such that $U=VX$.
	\end{enumerate}
	\end{theorem}
	As a result of the Douglas theorem, L. Gavruta gives  the following characterizations of $K$-frames in $H$.
	\begin{lemma}\label{Gavr1}\cite{3}
	Let $\{x_i\}_{i\in \mathcal{I}}$ be a Bessel sequence in $H$ whose the synthesis operator is $U$ and let $K\in B(H)$. Then:
	\begin{enumerate}
	\item $\{x_i\}_{i\in \mathcal{I}}$ is a $K$-frame $\;\;\Longleftrightarrow$ $R(K)\subset R(U)$.
	\item $\{x_i\}_{i\in \mathcal{I}}$ is a tight $K$-frame $\;\;\Longleftrightarrow$ $R(K)=R(U)$.
	\end{enumerate}
	\end{lemma}
	
	\begin{lemma}\label{Gavr2}\cite{3}
Let $\{x_i\}_{i\in \mathcal{I}}$ be a Bessel sequence for $H$ and $K\in B(H)$. The following statements are equivalent:
\begin{enumerate}
\item $\{x_i\}$ is a $K$-frame.
\item There exists a Bessel sequence $\{f_i\}_{i\in \mathcal{I}}$ in $H$ such that for all $x\in H$, $$Kx=\displaystyle{\sum_{i\in \mathcal{I}}\langle x,f_i\rangle x_i}.$$
Such a Bessel sequence is called a $K$-dual frame to $\{x_i\}_{i\in \mathcal{I}}$.\\
\end{enumerate}
\end{lemma}
	\section{M-D-G frames}
		In this section, we investigate M-D-G frames. We first give a matrix-characterization of frames in a particular class of  windows  and also get the explicit expression of the associated frame operator $S$ and its inverse $S^{-1}$ in terms of matrices $\mathcal{M}_{g_l}(.)\mathcal{M}_{g_l}^*(.)$.
	\begin{proposition}
	Let $g:=\{g_l\}_{l\in \mathbb{N}_L}\subset \ell^2(\mathbb{S})$ be such that $\vert supp(g_l)\vert< M$ for all $l\in \mathbb{N}_L$. Then the following statements are equivalent:
	\begin{enumerate}
	\item There exist $0< A\leq B<\infty$ such that for all $j\in \mathbb{S}\cap \mathbb{N}_N$, 
	$$A\leq M\displaystyle{\sum_{l\in \mathbb{N}_L} \left(\mathcal{M}_{g_l}(j)\mathcal{M}_{g_l}^*(j)\right)_{0,0}}\leq B.$$
	\item $\mathcal{G}(g,L,M,N)$ is a M-D-G frame for $\ell^2(\mathbb{S})$ with frame bounds $A$ and $B$.
	\end{enumerate}
In this case, we have the following explicit expressions of $S$ and $S^{-1}$: 
$$S(f)=M\displaystyle{\sum_{l\in \mathbb{N}_L} \left(\mathcal{M}_{g_l}(.)\mathcal{M}_{g_l}^*(.)\right)_{0,0}}f, \; \; S^{-1}f=\displaystyle{\frac{1}{M\displaystyle{\sum_{l\in \mathbb{N}_L} \left(\mathcal{M}_{g_l}(.)\mathcal{M}_{g_l}^*(.)\right)_{0,0}}}}f.$$
	\end{proposition} 
	
	\begin{proof}\hspace{1cm}
\begin{enumerate}
\item[]$(1)\Longrightarrow(2)$: Assume that there exist $0< A\leq B<\infty$ such that for all $j\in \mathbb{S}\cap \mathbb{N}_N$, 
	$$A\leq M\displaystyle{\sum_{l\in \mathbb{N}_L} \left(\mathcal{M}_{g_l}(j)\mathcal{M}_{g_l}^*(j)\right)_{0,0}}\leq B.$$
	By the lemma \ref{lem5}, we have for all $j\in \mathbb{S}$, $$A\leq M\displaystyle{\sum_{l\in \mathbb{N}_L} \left(\mathcal{M}_{g_l}(j)\mathcal{M}_{g_l}^*(j)\right)_{0,0}}\leq B.$$
	Let $f\in \ell_0(\mathbb{S})$. Since for all $l\in \mathbb{N}_L$, $\vert supp(g_l)\vert< M$, then for all $l\in \mathbb{N}_L$ and $k\in \mathbb{Z}-\{0\}$, $\left(\mathcal{M}_{g_l}(j)\mathcal{M}_{g_l}^*(j)\right)_{0,k}=0$. Then $F_2(f)=0$. Therefore: $$\begin{array}{rcl}

\displaystyle{\sum_{l\in \mathbb{N}_L}\sum_{n\in \mathbb{Z}}\sum_{m\in \mathbb{N}_M}\vert \langle f,E_{\frac{m}{M}}T_{nN}g_l\rangle \vert^2}&=&F_1(f)=M\displaystyle{\sum_{l\in \mathbb{N}_L}\sum_{j\in \mathbb{S}} \left(\mathcal{M}_{g_l}(j)\mathcal{M}_{g_l}^*(j)\right)_{0,0}\vert f(j)\vert^2}\\ 
&=& \displaystyle{\sum_{j\in \mathbb{S}} \; \left( M\,\sum_{l\in \mathbb{N}_L}\left(\mathcal{M}_{g_l}(j)\mathcal{M}_{g_l}^*(j)\right)_{0,0}\right)\vert f(j)\vert^2}\\ 
\end{array}$$
Then for all $f\in \ell_0(\mathbb{S})$, we have: 
$$ A\|f\|^2\leq \displaystyle{\sum_{l\in \mathbb{N}_L}\sum_{n\in \mathbb{Z}}\sum_{m\in \mathbb{N}_M}\vert \langle f,E_{\frac{m}{M}}T_{nN}g_l\rangle \vert^2}\leq B\|f\|^2.$$
By density of $\ell_0(\mathbb{S})$ in $\ell^2(\mathbb{S})$, we deduce that $\mathcal{G}(g,L,M,N)$ is a M-D-G frame with frame bounds $A$ and $B$.
\item[]$(2)\Longrightarrow (1)$: Theorem \ref{thm1}.
\end{enumerate}
Assume that one of these statements holds. Then for all $f\in \ell_0(\mathbb{S})$, we have: 
$$\begin{array}{rcl}
\langle Sf,f\rangle &=&F_1(f)\\
&=& M\displaystyle{\sum_{l\in \mathbb{N}_L}\sum_{j\in \mathbb{Z}} \left(\mathcal{M}_{g_l}(j)\mathcal{M}_{g_l}^*(j)\right)_{0,0} \vert f(j)\vert^2}\\
&=& \displaystyle{\sum_{j\in \mathbb{Z}} M\sum_{l\in \mathbb{N}_L}\left(\mathcal{M}_{g_l}(j)\mathcal{M}_{g_l}^*(j)\right)_{0,0} f(j) \,\overline{f(j)}}\\
&=& \langle \displaystyle{M\sum_{l\in \mathbb{N}_L}\left(\mathcal{M}_{g_l}(.)\mathcal{M}_{g_l}^*(.)\right)_{0,0} f}, f\rangle.
\end{array}$$
By density of $\ell_0(\mathbb{S})$ in $\ell^2(\mathbb{S})$ and since $f\mapsto  \displaystyle{M\sum_{l\in \mathbb{N}_L}\left(\mathcal{M}_{g_l}(.)\mathcal{M}_{g_l}^*(.)\right)_{0,0} f}\;\;$ is bounded and self-adjoint, then $$Sf=\displaystyle{M\sum_{l\in \mathbb{N}_L}\left(\mathcal{M}_{g_l}(.)\mathcal{M}_{g_l}^*(.)\right)_{0,0} f}.$$
And so $$S^{-1}f=\displaystyle{\frac{1}{M\displaystyle{\sum_{l\in \mathbb{N}_L} \left(\mathcal{M}_{g_l}(.)\mathcal{M}_{g_l}^*(.)\right)_{0,0}}}}f.$$
    \end{proof}
    
We have the following necessary condition on the parameters $M,N$ and $L$ for a M-D-G system in $\ell^2(\mathbb{S})$ to be frame. We also characterize which M-D-G frames are M-D-G Riesz bases.
\begin{theorem} \label{Rieszbasis}\hspace{1cm}
\begin{enumerate}
\item  A M-D-G system $\mathcal{G}(g,L,M,N)$ is a frame for $\ell^2(\mathbb{S})$ only when $\displaystyle{\frac{card(\mathbb{S}_N)}{M}}\leq L$.
\item  Assume that $\mathcal{G}(g,L,M,N)$ is a frame for $\ell^2(\mathbb{S})$. Then following statements are equivalent: 
\begin{enumerate}
\item $\mathcal{G}(g,L,M,N)$ is a M-D-G Riesz basis (M-D-G exact frame) for $\ell^2(\mathbb{S})$.
\item $\displaystyle{\frac{card(\mathbb{S}_N)}{M}}=L$.
\end{enumerate}
\end{enumerate}
\end{theorem}

\begin{proof} Assume that $\mathcal{G}(g,L,M,N)$ is a M-D-G frame for $\ell^2(\mathbb{S})$. Then, by lemma \ref{lem2}, $\mathcal{G}(S^{\frac{-1}{2}}g,L,M,N)$ is a  M-D-G Parseval frame for $\ell^2(\mathbb{S})$, where $S$ is the frame operator of $\mathcal{G}(g,L,M,N)$ and $S^{\frac{-1}{2}}g:=\{S^{\frac{-1}{2}}g_l\}_{l\in \mathbb{N}_L}$. 
Then:
\begin{enumerate}
\item  We can suppose that $\mathcal{G}(g,L,M,N)$ is a Parseval frame (because otherwise, we can consider $\mathcal{G}(S^{\frac{-1}{2}}
g,L,M,N)$ instead of $\mathcal{G}(g,L,M,N)$). By the theorem \ref{thm1}, we have for all $j\in \mathbb{S}_N$,  $\displaystyle{\sum_{l\in \mathbb{N}_L} \left(\mathcal{M}_{g_l}(j)\mathcal{M}_{g_l}^*(j)\right)_{0,0}=\frac{1}{M}}.$ Then $\displaystyle{\sum_{l\in \mathbb{N}_L} \sum_{n\in \mathbb{Z}}\vert g_l(j-nN)\vert^2=\frac{1}{M}}$. 
Then $\displaystyle{\sum_{j\in \mathbb{S}_N}\sum_{l\in \mathbb{N}_L} \sum_{n\in \mathbb{Z}}\vert g_l(j-nN)\vert^2=\frac{card(\mathbb{S}_N)}{M}}$.\\
Hence $\displaystyle{\sum_{l\in \mathbb{N}_L}\sum_{j\in \mathbb{S}_N} \sum_{n\in \mathbb{Z}}\vert g_l(j-nN)\vert^2=\frac{card(\mathbb{S}_N)}{M}}$.
Hence $\displaystyle{\sum_{l\in \mathbb{N}_L}\|g_l\|^2=\frac{card(\mathbb{S}_N)}{M}}$. \\Since $\mathcal{G}(g,L,M,N)$ is a Parseval frame, then for all $l\in \mathbb{N}_L$,  $\|E_{\frac{m}{M}}T_{nN}g_l\|=\|g_l\|\leq 1$. Hence $\displaystyle{\frac{card(\mathbb{S}_N)}{M}}\leq L$.
\item Assume that $\mathcal{G}(g,L,M,N)$ is a frame for $\ell^2(\mathbb{S})$. Then $$ \displaystyle{\sum_{l\in \mathbb{N}_L}\|S^{\frac{-1}{2}} g_l\|^2=\frac{card(\mathbb{S}_N)}{M}}.$$ 
\begin{enumerate}
\item[]$\Longrightarrow$) We have:
$$\begin{array}{rcl}
&&\mathcal{G}(g,L,M,N) \text{ is a Riesz basis for } \ell^2(\mathbb{S})\\ &\Longrightarrow& \mathcal{G}(S^{\frac{-1}{2}}g, L,M,N) \text{ is an orthonormal basis for }  \ell^2(\mathbb{S}).\\
&\Longrightarrow& \displaystyle{\sum_{l\in \mathbb{N}_L}\|S^{\frac{-1}{2}} g_l\|^2=\frac{card(\mathbb{S}_N)}{M}} \text{ and } \|S^{\frac{-1}{2}}g_l\|=1.\\
&\Longrightarrow& \displaystyle{\frac{card(\mathbb{S}_N)}{M}=L}
\end{array}$$
\item[]$\Longleftarrow$) Conversely, Assume that $\displaystyle{\frac{card(\mathbb{S}_N)}{M}=L}$. Then $ \displaystyle{\sum_{l\in \mathbb{N}_L}\|S^{\frac{-1}{2}} g_l\|^2=L}.$ Since for all $l\in \mathbb{N}_L$,  $\|S^{\frac{-1}{2}}g_l\|\leq 1$, then for all $l\in \mathbb{N}_L$, $\|S^{\frac{-1}{2}}g_l\|=1$. Hence $\mathcal{G}(S^{\frac{-1}{2}}g,L,M,N)$ is an orthonormal basis for $\ell^2(\mathbb{S})$ (lemma \ref{lem3'}). Hence, by lemma \ref{lem2}, $\mathcal{G}(g,L,M,N)$ is a Riesz basis.
\end{enumerate}
\end{enumerate}
\end{proof}
The following result is a matrix-characterization of  Parseval M-D-G frames for $\ell^2(\mathbb{S})$.
\begin{theorem}\label{pars}
Let $\mathcal{G}(g,L,M,N)$ be a M-D-G system for $\ell^2(\mathbb{S})$. Then the following statements are equivalent:
\begin{enumerate}
\item $\mathcal{G}(g,L,M,N)$ is a  M-D-G Parseval frame for $\ell^2(\mathbb{S})$.
\item \begin{enumerate}
     \item For all $j\in \mathbb{S}_N$, $\displaystyle{\sum_{l\in \mathbb{N}_L}\left(\mathcal{M}_{g_l}(j)\mathcal{M}_{g_l}^*(j)\right)_{0,0}=\frac{1}{M}}$.
     \item For all $j\in \mathbb{S}_N$, $k\in \mathbb{Z}-\{0\}$, $\displaystyle{\sum_{l\in \mathbb{N}_L}\left(\mathcal{M}_{g_l}(j)\mathcal{M}_{g_l}^*(j)\right)_{0,k}}=0$.
\end{enumerate}
\end{enumerate}
\end{theorem}

\begin{proof}\hspace{1cm}
\begin{enumerate}
\item[]$\Longrightarrow)$ Assume that  $\mathcal{G}(g,L,M,N)$ is a Parseval M-D-G frame. Then by theorem \ref{thm1}, we have for all $j\in \mathbb{S}_N$, $\displaystyle{\sum_{l\in \mathbb{N}_L}\left(\mathcal{M}_{g_l}(j)\mathcal{M}_{g_l}^*(j)\right)_{0,0}=\frac{1}{M}}$, hence (a). On the other hand, fix $k_0 \in \mathbb{Z}-\{0\}$ and denote $a_{k_0}(.):=\left(\displaystyle{\sum_{l\in \mathbb{N}_L}\mathcal{M}_{g_l}(.)\mathcal{M}_{g_l}^*(.)}\right)_{0,k_0}$. Define $f\in \ell_0(\mathbb{S})$ as follows $ f(j)=e^{-i \,arg(a_{k_0(j)})}$ and $f(j+kM)=1$ for all $j\in \mathbb{S}_N, \; k\in \mathbb{Z}-\{0\}$ and $f(j)=0$  otherwise.\\ Using (a), we have $F_1(f)=\|f\|^2$. Then: $$\|f\|^2= \displaystyle{\sum_{l\in \mathbb{N}_L}\sum_{n\in \mathbb{Z}}\sum_{m\in \mathbb{N}_M}\left \vert \langle f, E_{\frac{m}{M}}T_{nN}g_l\rangle \right\vert^2}=F_1(f)+F_2(f)=\|f\|^2+F_2(f).$$
Then $F_2(f)=0$. Thus $M\displaystyle{\sum_{j\in \mathbb{S}_N} \vert a_{k_0(j)}\vert^2}=0$. Then for all $j\in \mathbb{S}_N$, $a_{k_0}(j)=0$.\\
That means that for all $j\in \mathbb{S_N}$,  $\left(\displaystyle{\sum_{l\in \mathbb{N}_L}\mathcal{M}_{g_l}(j)\mathcal{M}_{g_l}^*(j)}\right)_{0,k_0}=0$.\\
Hence For all $j\in \mathbb{S}_N$, $k\in \mathbb{Z}-\{0\}$, $\displaystyle{\sum_{l\in \mathbb{N}_L}\left(\mathcal{M}_{g_l}(j)\mathcal{M}_{g_l}^*(j)\right)_{0,k}}=0$.
\item[]$\Longleftarrow$) Conversely, assume (a) and (b). For all $f\in \ell_0(\mathbb{S})$, we have: 
$$\begin{array}{rcl}
\displaystyle{\sum_{l\in \mathbb{N}_L}\sum_{n\in \mathbb{Z}}\sum_{m\in \mathbb{N}_M}\left \vert \langle f, E_{\frac{m}{M}}T_{nN}g_l\rangle \right\vert^2}&=&F_1(f)+F_2(f)\\
&=&\|f\|^2+0\\
&=& \|f\|^2.

\end{array}$$
By density of $\ell_0(\mathbb{S})$ in $\ell^2(\mathbb{S})$, we deduce that $\mathcal{G}(g,L,M,N)$ is a M-D-G  Parseval frame for $\ell^2(\mathbb{S})$.
\end{enumerate}
\end{proof}	

\begin{corollary}\label{parsevalsupp}
Let $\mathcal{G}(g,L,M,N)$ be a M-D-G system for $\ell^2(\mathbb{S})$ such that for all $l\in \mathbb{N}_L$, $\vert supp(g_l)\vert < M$. Then the following statements are equivalent:
\begin{enumerate}
\item $\mathcal{G}(g,L,M,N)$ is a M-D-G Parseval frame for $\ell^2(\mathbb{S})$.\\
\item  For all $j\in \mathbb{S}_N$, $\displaystyle{\sum_{l\in \mathbb{N}_L}\left(\mathcal{M}_{g_l}(j)\mathcal{M}_{g_l}^*(j)\right)_{0,0}=\frac{1}{M}}$.
\end{enumerate}
\end{corollary}

\begin{proof}
Since for all $l\in \mathbb{N}_L$, $\vert supp(g_l)\vert < M$, then For all $j\in \mathbb{S}_N$, $k\in \mathbb{Z}-\{0\}$, $\displaystyle{\sum_{l\in \mathbb{N}_L}\left(\mathcal{M}_{g_l}(j)\mathcal{M}_{g_l}^*(j)\right)_{0,k}}=0$. Hence, Theorem \ref{pars} completes the proof.
\end{proof}

In the  following theorem, we characterize  the existence of M-D-G (Parseval) frames for $\ell^2(\mathbb{Z})$ using the parameters $M$, $N$ and $L$. Our proof is constructive.
\begin{theorem}\label{exframe}
Let $L,M,N\in \mathbb{N}$. Then the following statements are equivalent:
\begin{enumerate}
\item $N\leq LM$.
\item There exists  $g:=\{g_l\}_{l\in \mathbb{N}_L}\subset \ell^2(\mathbb{Z})$ such that $\mathcal{G}(g,L,M,N)$ is a M-D-G Parseval frame for $\ell^2(\mathbb{Z})$.
\item There exists  $g:=\{g_l\}_{l\in \mathbb{N}_L}\subset \ell^2(\mathbb{Z})$ such that $\mathcal{G}(g,L,M,N)$ is a M-D-G frame for $\ell^2(\mathbb{Z})$.
\end{enumerate}
\end{theorem}

\begin{proof}\hspace{1cm}
\begin{enumerate}
\item[]$1\Longrightarrow2)$ Assume that $N\leq LM$. Let $K\in \mathbb{N}$ be the maximal integer satisfying $KL\leq N$.  Define $\mathbb{N}_N^0$ as the set of the first $M$ elements of $\mathbb{N}_N$ and then for all $l\in \mathbb{N}_K: l\geqslant1$, $\mathbb{N}_N^l$ as the set of the $(l+1)$-th $M$ elements of $\mathbb{N}_N$ and $\mathbb{N}_N^K$ as the set of the rest elements of $\mathbb{N}_N$. Define for all $L>l>K$, $g_l:=0$ and for $l\in \mathbb{N}_{K+1}$, $g_l:=\chi_{\mathbb{N}_N^l}.\displaystyle{\frac{1}{\sqrt{M}}}$. It is clear that $g_l\in \ell_0(\mathbb{Z})$, $\forall l\in \mathbb{N}_{K+1},\; supp(g_l)=\mathbb{N}_N^l$, $\displaystyle{\bigcup_{l\in \mathbb{N}_{K+1}} \mathbb{N}_N^l}=\mathbb{N}_N$ and $\mathbb{N}_N^l\bigcap \mathbb{N}_N^{l'}=\emptyset$ if $l,l'\in \mathbb{N}_{K+1}$ and $l\neq l'$. Since for all $l\in \mathbb{N}_L$, $\vert supp(g_l)\vert< M$, then to show that $\mathcal{G}(g,L,M,N)$, where $g=\{g_l\}_{l\in \mathbb{N}_L}$, is a M-D-G Parseval frame for $\ell^2(\mathbb{Z})$, it suffices to show the condition (2)  in corollary \ref{parsevalsupp}. Let $j\in \mathbb{N}_N$, then there exists a unique $l_0\in \mathbb{N}_{K+1}$ such that $j\in \mathbb{N}_N^{l_0}$. Then we have
$$\begin{array}{rcl}
\displaystyle{\left(\sum_{l\in \mathbb{N}_L}\mathcal{M}_{g_l}(j)\mathcal{M}_{g_l}^*(j)\right)_{0,0}}&=&\displaystyle{\sum_{l\in \mathbb{N}_L}\sum_{n\in \mathbb{Z}}\vert g_l(j-nN)\vert^2}\\
&=&\displaystyle{\sum_{l\in \mathbb{N}_{K+1}}\sum_{n\in \mathbb{Z}}\vert g_l(j-nN)\vert^2}\\
&=& \displaystyle{\sum_{n\in \mathbb{Z}}\vert g_{l_0}(j-nN)\vert^2}\\
&=&\displaystyle{\vert g_{l_0}(j)\vert^2}\\
&=& \displaystyle{\frac{1}{M}}.
\end{array}$$
Hence, by corollary \ref{parsevalsupp}, $\mathcal{G}(g,L,M,N)$ is a M-D-G Parseval frame for $\ell^2(\mathbb{Z})$.
\item[]$2\Longrightarrow3)$ Clear.
\item[]$3\Longrightarrow 1$) By proposition \ref{Rieszbasis}.\\
\end{enumerate}
\end{proof}

The following example is a direct consequence of the above proof.
\begin{example}
Let $N=3$, $M=2$, $L=2$. Define $g_0=\chi_{\{0,1\}}.\displaystyle{\frac{1}{\sqrt{2}}}$ and $g_1=\chi_{\{2\}}.\displaystyle{\frac{1}{\sqrt{2}}}$. Then $\mathcal{G}(\,\{g_0,g_1\}\, ,2,2,3)$ is a two-window D-G Parseval frame for $\ell^2(\mathbb{Z})$.\\
\end{example}
The following result is a matrix-characterization of M-D-G orthonormal bases for $\ell^2(\mathbb{S})$.

\begin{proposition}\label{ortbasis}
Let $\mathcal{G}(g,L,M,N)$ be a M-D-G system for $\ell^2(\mathbb{S})$. Then the following statements are equivalent:
\begin{enumerate}
\item $\mathcal{G}(g,L,M,N)$ is a  M-D-G orthonormal basis for $\ell^2(\mathbb{S})$.
\item \begin{enumerate}
\item For all $l\in \mathbb{N}_L$, $\|g_l\|=1$.
     \item For all $j\in \mathbb{S}_N$, $\displaystyle{\sum_{l\in \mathbb{N}_L}\left(\mathcal{M}_{g_l}(j)\mathcal{M}_{g_l}^*(j)\right)_{0,0}=\frac{1}{M}}$.
     \item For all $j\in \mathbb{S}_N$, $k\in \mathbb{Z}-\{0\}$, $\displaystyle{\sum_{l\in \mathbb{N}_L}\left(\mathcal{M}_{g_l}(j)\mathcal{M}_{g_l}^*(j)\right)_{0,k}}=0$.
\end{enumerate}
\end{enumerate}
\end{proposition} 

\begin{proof}
Since the $E_{\frac{m}{M}}T_{nN}$ are unitaries, Lemma \ref{lem3'} and Theorem \ref{pars} complete the proof.
\end{proof} 
\begin{corollary}\label{orthosupp}
Let $\mathcal{G}(g,L,M,N)$ be a M-D-G system for $\ell^2(\mathbb{S})$ such that for all $l\in \mathbb{N}_L$, $\vert supp(g_l)\vert < M$. Then the following statements are equivalent:
\begin{enumerate}
\item $\mathcal{G}(g,L,M,N)$ is a M-D-G orthonormal basis for $\ell^2(\mathbb{S})$.
\item \begin{enumerate}
\item For all $l\in \mathbb{N}_L$, $\|g_l\|=1$.
\item For all $j\in \mathbb{S}_N$, $\displaystyle{\sum_{l\in \mathbb{N}_L}\left(\mathcal{M}_{g_l}(j)\mathcal{M}_{g_l}^*(j)\right)_{0,0}=\frac{1}{M}}$.
\end{enumerate}
\end{enumerate}
\end{corollary}
\begin{proof}
Since for all $l\in \mathbb{N}_L$, $\vert supp(g_l)\vert < M$, then For all $j\in \mathbb{S}_N$, $k\in \mathbb{Z}-\{0\}$, $\displaystyle{\sum_{l\in \mathbb{N}_L}\left(\mathcal{M}_{g_l}(j)\mathcal{M}_{g_l}^*(j)\right)_{0,k}}=0$. Hence, proposition \ref{ortbasis} completes the proof.
\end{proof}
The following result is an other matrix-characterization, more optimal in practice,  of M-D-G orthonormal bases for $\ell^2(\mathbb{S})$ using the parameters $M$, $N$ and $L$.

\begin{proposition}\label{ortbasis2}
Let $\mathcal{G}(g,L,M,N)$ be a M-D-G system for $\ell^2(\mathbb{S})$. Then the following statements are equivalent:
\begin{enumerate}
\item $\mathcal{G}(g,L,M,N)$ is a  M-D-G orthonormal basis.
\item \begin{enumerate}
\item $card(\mathbb{S}_N)=LM.$
\item For all $j\in \mathbb{S}_N$, $\displaystyle{\sum_{l\in \mathbb{N}_L}\left(\mathcal{M}_{g_l}(j)\mathcal{M}_{g_l}^*(j)\right)_{0,0}=\frac{1}{M}}$.
\item For all $j\in \mathbb{S}_N$, $k\in \mathbb{Z}-\{0\}$, $\displaystyle{\sum_{l\in \mathbb{N}_L}\left(\mathcal{M}_{g_l}(j)\mathcal{M}_{g_l}^*(j)\right)_{0,k}}=0$.
\end{enumerate}
\end{enumerate}
\end{proposition}
\begin{proof}

Since $\mathcal{G}(g,L,M,N)$ is a  M-D-G orthonormal basis for $\ell^2(\mathbb{S})$ if and only if $\mathcal{G}(g,L,M,N)$ is  a Riesz basis and a Parseval frame for $\ell^2(\mathbb{S})$, 
then proposition \ref{Rieszbasis} and proposition \ref{pars} complete the proof.

\end{proof}
 \begin{corollary}\label{orthosupp2}
Let $\mathcal{G}(g,L,M,N)$ be a M-D-G system for $\ell^2(\mathbb{S})$ such that for all $l\in \mathbb{N}_L$, $\vert supp(g_l)\vert < M$. Then the following statements are equivalent:
\begin{enumerate}
\item $\mathcal{G}(g,L,M,N)$ is a M-D-G orthonormal basis.\\
\item \begin{enumerate}
\item $card(\mathbb{S}_N)=LM$.
\item For all $j\in \mathbb{S}_N$, $\displaystyle{\sum_{l\in \mathbb{N}_L}\left(\mathcal{M}_{g_l}(j)\mathcal{M}_{g_l}^*(j)\right)_{0,0}=\frac{1}{M}}$.
\end{enumerate}
\end{enumerate}
\end{corollary}

\begin{proof}
Since for all $l\in \mathbb{N}_L$, $\vert supp(g_l)\vert < M$, then For all $j\in \mathbb{S}_N$, $k\in \mathbb{Z}-\{0\}$, $\displaystyle{\sum_{l\in \mathbb{N}_L}\left(\mathcal{M}_{g_l}(j)\mathcal{M}_{g_l}^*(j)\right)_{0,k}}=0$. Hence, proposition \ref{ortbasis2} completes the proof.

\end{proof}
In the following proposition, we characterize the existence of M-D-G Riesz bases and  M-D-G orthonormal bases for $\ell^2(\mathbb{Z})$ using the parameters $M$, $N$ and $L$. The proof we provide is constructive.
\begin{theorem}\label{exortho}
Let  $L,N,M\in \mathbb{N}$. Then the following statements are equivalent:
\begin{enumerate}
\item $N=LM$.
\item There exists $g=\{g_l\}_{l\in \mathbb{N}_L}\subset \ell^2(\mathbb{Z})$ such that $\mathcal{G}(g,L,M,N)$ is a M-D-G orthonormal basis for $\ell^2(\mathbb{Z})$.
\item There exists $g=\{g_l\}_{l\in \mathbb{N}_L}\subset \ell^2(\mathbb{S})$ such that $\mathcal{G}(g,L,M,N)$ is a M-D-G Riesz basis for $\ell^2(\mathbb{Z})$.

\end{enumerate}
\end{theorem}
\begin{proof}\hspace{1cm}
\begin{enumerate}
\item[]$1\Longrightarrow2)$ Assume that $N=LM$. Define  $\mathbb{N}_N^0$ as the set of the first $M$ elements of $\mathbb{N}_N$ and then for all $l\in \mathbb{N}_L: l\geqslant1$, $\mathbb{N}_N^l$ is the set of the $(l+1)$-th $M$ elements of $\mathbb{N}_N$. Define for each $l\in \mathbb{N}_L$, $g_l:=\chi_{\mathbb{S}_N^l}\,. \, \displaystyle{\frac{1}{\sqrt{M}}}$. It is clear that $g_l\in \ell_0(\mathbb{Z})$, $supp(g_l)=\mathbb{N}_N^l$, $\displaystyle{\bigcup_{l\in \mathbb{N}_L} \mathbb{N}_N^l}=\mathbb{N}_N$ and $\mathbb{N}_N^l\bigcap \mathbb{N}_N^{l'}=\emptyset$ if $l\neq l'$. Since $\vert \mathbb{N}_N^l\vert=M-1< M$, then to show that $\mathcal{G}(g,L,M,N)$, where $g=\{g_l\}_{l\in \mathbb{N}_L}$, is a M-D-G orthonormal basis for $\ell^2(\mathbb{Z})$, it suffices to show the condition (b) in corollary \ref{orthosupp2}. \\
Let $j\in \mathbb{N}_N$, then there exists a unique $l_0\in \mathbb{N}_L$ such that  $j\in \mathbb{N}_N^{l_0}$. Then we have:
$$\begin{array}{rcl}
\displaystyle{\left(\sum_{l\in \mathbb{N}_L}\mathcal{M}_{g_l}(j)\mathcal{M}_{g_l}^*(j)\right)_{0,0}}&=&\displaystyle{\sum_{l\in \mathbb{N}_L}\sum_{n\in \mathbb{Z}}\vert g_l(j-nN)\vert^2}\\
&=& \displaystyle{\sum_{n\in \mathbb{Z}}\vert g_{l_0}(j-nN)\vert^2}\\
&=&\vert g_{l_0}(j)\vert^2=\displaystyle{\frac{1}{M}}.
\end{array}$$
Hence, by corollary \ref{orthosupp2}, $\mathcal{G}(g,L,M,N)$ is a M-D-G orthonormal basis for $\ell^2(\mathbb{Z})$.
\item[]$2\Longrightarrow 3)$ Since every orthonormal basis is a Riesz basis.
\item[]$3\Longrightarrow1)$  By Theorem \ref{Rieszbasis}.\\
\end{enumerate}
\end{proof}
The following example is a direct consequence of the above proof.
\begin{example}
Let $N=12$, $M=4$, $L=3$. Define $g_0=\chi_{\{0,\,1,\,2,\,3\}}.\displaystyle{\frac{1}{2}}$,  $g_1=\displaystyle{\chi_{\{4,\,5,\, 6,\, 7\}}.\frac{1}{2}}$ and $g_2=\displaystyle{\chi_{\{8,\,9,\,10,\,11\}}.\frac{1}{2}}$. Then $\mathcal{G}(\, \{g_0,\,g_1,g_2\}\,,3,4,12)$ is a three-window D-G orthonormal basis for $\ell^2(\mathbb{Z})$.\\
\end{example}

There is a strong connection between Parseval frames and Duality; in fact a  Parseval frame is exactely a frame that is dual with itself. The following result generalizes the theorem \ref{pars} to characterize dual frames in $\ell^2(\mathbb{S})$.
\begin{theorem}\label{dual}
Let $\mathcal{G}(g,L,M,N)$ and $\mathcal{G}(h,L,M,N)$  be two M-D-G Bessel systems in $\ell^2(\mathbb{S})$. Then the following statements are equivalent:
\begin{enumerate}
\item $\mathcal{G}(g,L,M,N)$ and $\mathcal{G}(h,L,M,N)$ are dual M-D-G frames for $\ell^2(\mathbb{S})$.
\item \begin{enumerate}
     \item For all $j\in \mathbb{S}_N$, $\displaystyle{\sum_{l\in \mathbb{N}_L}\left(\mathcal{M}_{h_l}(j)\mathcal{M}_{g_l}^*(j)\right)_{0,0}=\frac{1}{M}}$.
     \item For all $j\in \mathbb{S}_N$, $k\in \mathbb{Z}-\{0\}$, $\displaystyle{\sum_{l\in \mathbb{N}_L}\left(\mathcal{M}_{h_l}(j)\mathcal{M}_{g_l}^*(j)\right)_{0,k}}=0$.\\
\end{enumerate}
\end{enumerate}
\end{theorem}
\begin{proof}\hspace{1cm}
\begin{enumerate}
\item[]$(1)\Longrightarrow (2):$ Assume that $\mathcal{G}(g,L,M,N)$ and $\mathcal{G}(h,L,M,N)$ are dual M-D-G frames. Then for all $f\in \ell_0(\mathbb{S})$, we have: $$\begin{array}{rcl}

\|f\|^2=\langle f,f\rangle&=&\displaystyle{\sum_{l\in \mathbb{N}_L}\sum_{n\in \mathbb{Z}}\sum_{m\in \mathbb{N}_M}\langle f,E_{\frac{m}{M}}T_{nN}g_l\rangle\,\langle E_{\frac{m}{M}}T_{nN}h_l,f\rangle}\\
	&=&M\displaystyle{\sum_{j\in \mathbb{S}}\sum_{k\in \mathbb{Z}}\left(\sum_{l\in \mathbb{N}_L}\mathcal{M}_{h_l}(j)\mathcal{M}_{g_l}^*(j)\right)_{0,k}\,f(j+kM)\overline{f(j)}}
\end{array}$$
Take any $f\in \ell_0(\mathbb{S})$ such that $\vert supp(f)\vert< M$, then for $k\neq 0$, $f(j+kM)\overline{f(j)}=0$. Thus:
$$\|f\|^2=M\displaystyle{\sum_{j\in \mathbb{S}}\left(\sum_{l\in \mathbb{N}_L}\mathcal{M}_{h_l}(j)\mathcal{M}_{g_l}^*(j)\right)_{0,0}\,\vert f(j)\vert^2}$$
Fix $j_0\in \mathbb{S}$ and define $f_0:=\chi_{j_0}$. Then:
$$1=M\left(\sum_{l\in \mathbb{N}_L}\mathcal{M}_{h_l}(j_0)\mathcal{M}_{g_l}^*(j_0)\right)_{0,0}.$$
Hence for all $j_0\in \mathbb{S}$, $\displaystyle{\left(\sum_{l\in \mathbb{N}_L}\mathcal{M}_{h_l}(j_0)\mathcal{M}_{g_l}^*(j_0)\right)_{0,0}=\frac{1}{M}}$.\\
On the other hand, we have for all $f\in \ell_0(\mathbb{S})$, $$\begin{array}{rcl}
\|f\|^2=\langle f,f\rangle&=&\displaystyle{\sum_{l\in \mathbb{N}_L}\sum_{n\in \mathbb{Z}}\sum_{m\in \mathbb{N}_M}\langle f,E_{\frac{m}{M}}T_{nN}g_l\rangle\,\langle E_{\frac{m}{M}}T_{nN}h_l,f\rangle}\\
&=&M\displaystyle{\sum_{j\in \mathbb{S}}\sum_{k\in \mathbb{Z}}\left(\sum_{l\in \mathbb{N}_L}\mathcal{M}_{h_l}(j)\mathcal{M}_{g_l}^*(j)\right)_{0,k}\,f(j+kM)\overline{f(j)}}\\
&=&\|f\|^2+M\displaystyle{\sum_{j\in \mathbb{S}}\sum_{k\in \mathbb{Z}-\{0\}}\left(\sum_{l\in \mathbb{N}_L}\mathcal{M}_{h_l}(j)\mathcal{M}_{g_l}^*(j)\right)_{0,k}\,f(j+kM)\overline{f(j)}}.
\end{array}$$
Thus for all $f\in \ell_0(\mathbb{S})$, $$\displaystyle{\sum_{j\in \mathbb{S}}\sum_{k\in \mathbb{Z}-\{0\}}\left(\sum_{l\in \mathbb{N}_L}\mathcal{M}_{h_l}(j)\mathcal{M}_{g_l}^*(j)\right)_{0,k}\,f(j+kM)\overline{f(j)}}
=0.$$
Fix $k_0\in \mathbb{Z}-\{0\}$. Define $f\in \ell_0(\mathbb{S}$ as follows:  $f_0(j)=e^{-i arg(a_0(j))}$ and $f(j)=1$ for all $j\in \mathbb{S}_N
$, where $a_0(j)=\displaystyle{\left(\sum_{l\in \mathbb{N}_L}\mathcal{M}_{h_l}(j)\mathcal{M}_{g_l}^*(j)\right)_{0,k_0}}$. Then: 
$$\begin{array}{rcl}
0&=&\displaystyle{\sum_{j\in \mathbb{S}}\sum_{k\in \mathbb{Z}-\{0\}}\left(\sum_{l\in \mathbb{N}_L}\mathcal{M}_{h_l}(j)\mathcal{M}_{g_l}^*(j)\right)_{0,k}\,f_0(j+kM)\overline{f_0(j)}}\\
&=&\displaystyle{\sum_{j\in \mathbb{S}_N}\vert a_0(j)\vert^2}.
\end{array}$$
Then for all $j\in \mathbb{S}_N$, $a_0(j):=\displaystyle{\left(\sum_{l\in \mathbb{N}_L}\mathcal{M}_{h_l}(j)\mathcal{M}_{g_l}^*(j)\right)_{0,k_0}}=0$.\\
Hence for all $j\in \mathbb{S}_N$, $k\in \mathbb{Z}-\{0\}$, $\displaystyle{\left(\sum_{l\in \mathbb{N}_L}\mathcal{M}_{h_l}(j)\mathcal{M}_{g_l}^*(j)\right)_{0,k}}=0$.\\
\item[]$(2)\Longrightarrow (1):$ Assume (a) and (b). Then, by lemma \ref{lem4'}, for all $f,\phi \in \ell_0(\mathbb{S})$, we have: 
$$\begin{array}{rcl}
&&\displaystyle{\sum_{l\in \mathbb{N}_L}\sum_{n\in \mathbb{Z}}\sum_{m\in \mathbb{N}_M}\langle f,E_{\frac{m}{M}}T_{nN}g_l\rangle\,\langle E_{\frac{m}{M}}T_{nN}h_l,\phi\rangle}\\
	&=&M\displaystyle{\sum_{j\in \mathbb{S}}\sum_{k\in \mathbb{Z}}\left(\sum_{l\in \mathbb{N}_L}\mathcal{M}_{h_l}(j)\mathcal{M}_{g_l}^*(j)\right)_{0,k}\,f(j+kM)\overline{\phi(j)}}\\
	&=&M\displaystyle{\sum_{j\in \mathbb{S}}\left(\sum_{l\in \mathbb{N}_L}\mathcal{M}_{h_l}(j)\mathcal{M}_{g_l}^*(j)\right)_{0,k}\,f(j)\overline{\phi(j)}}\\
	&+&M\displaystyle{\sum_{j\in \mathbb{S}}\sum_{k\in \mathbb{Z}-\{0\}}\left(\sum_{l\in \mathbb{N}_L}\mathcal{M}_{h_l}(j)\mathcal{M}_{g_l}^*(j)\right)_{0,k}\,f(j+kM)\overline{\phi(j)}}\\
&=&\langle f,\phi\rangle\;\;\; (\text{ by } (a) \text{ and } (b)\;).
\end{array}$$
Then, by density of $\ell_0(\mathbb{S})$ in $\ell^2(\mathbb{S})$, we heve for all $f,\phi \in \ell^2(\mathbb{S})$, 
$$\displaystyle{\sum_{l\in \mathbb{N}_L}\sum_{n\in \mathbb{Z}}\sum_{m\in \mathbb{N}_M}\langle f,E_{\frac{m}{M}}T_{nN}g_l\rangle\,\langle E_{\frac{m}{M}}T_{nN}h_l,\phi\rangle}=\langle f,\phi\rangle.$$
Thus for all $f\in \ell^2(\mathbb{S})$, $$f=\displaystyle{\sum_{l\in \mathbb{N}_L}\sum_{n\in \mathbb{Z}}\sum_{m\in \mathbb{N}_M}\langle f,E_{\frac{m}{M}}T_{nN}g_l\rangle\, E_{\frac{m}{M}}T_{nN}h_l}.$$
Since  $\mathcal{G}(g,L,M,N)$ and $\mathcal{G}(h,L,M,N)$ are Bessel sequences, then they are  dual M-D-G frames in $\ell^2(\mathbb{S})$.
\end{enumerate}

\end{proof}

\begin{corollary}
Let $g:=\{g_l\}_{l\in \mathbb{N}_L},\;h:=\{h_l\}_{l\in \mathbb{N}_L}\subset \ell^2(\mathbb{S})$ such that for all $l\in \mathbb{N}_L$, $\vert supp(g_l)\cup supp(h_l)\vert< M$. Then the following statements are equivalent: 
\begin{enumerate}
\item $\mathcal{G}(g,L,M,N)$ and $\mathcal{G}(h,L,M,N)$ are dual M-D-G frames.
\item For all $j\in \mathbb{S}_N$, $\displaystyle{\sum_{l\in \mathbb{N}_L}\left(\mathcal{M}_{g_l}(j)\mathcal{M}_{h_l}^*(j)\right)_{0,0}=\frac{1}{M}}$.\\
\end{enumerate}
\end{corollary}
\begin{proof} Since $g_l,h_l\in \ell_0(\mathbb{S})$, then $\mathcal{G}(g,L,M,N)$ and $\mathcal{G}(h,L,M,N)$ are M-D-G Bessel systems.
\begin{enumerate}
\item[]$(1)\Longrightarrow(2):$ Theorem \ref{dual}.
\item[]$(2)\Longrightarrow (1):$ Assume that for all $j\in \mathbb{S}_N$, $\displaystyle{\sum_{l\in \mathbb{N}_L}\left(\mathcal{M}_{g_l}(j)\mathcal{M}_{h_l}^*(j)\right)_{0,0}=\frac{1}{M}}$.\\
Since for all $l\in \mathbb{N}_L$, $\vert supp(g_l)\cup supp(h_l)\vert< M$, then for all $j\in \mathbb{S}_N$, $k\neq 0$,  $$\begin{array}{rcl}
\displaystyle{\left(\sum_{l\in \mathbb{N}_L}\mathcal{M}_{h_l}(j)\mathcal{M}_{g_l}^*(j)\right)_{0,k}}&=&\displaystyle{\sum_{n\in \mathbb{Z}}h_l(j-nN)\,\overline{g_l(j+kM-nN})}\\
&=&\displaystyle{\sum_{n\in \mathbb{Z}}h_l(j-nN)\,\overline{g_l(j-nN+kM})}\\
&=&0.
\end{array}$$
Then, Theorem \ref{dual} completes the proof.
\end{enumerate}
\end{proof}
The next result is a perturbation condition of M-D-G frames.
	\begin{theorem}
	Assume that $\mathcal{G}(g,L,M,N)$ is a M-D-G frame in $\ell^2(\mathbb{S})$ with frame bounds $A$ and $B$. Let $h:=\{h_l\}_{l\in \mathbb{N}_L}\subset \ell^2(\mathbb{S})$ and suppose that: 
	$$R:=M\,\displaystyle{\max_{j\in \mathbb{S}_N}\sum_{k\in \mathbb{Z}}\left\vert \left(\sum_{l\in \mathbb{N}_L}\mathcal{M}_{g_l-h_l}(j)\mathcal{M}_{g_l-h_l}^*(j)\right)_{0,k}\right\vert}< A.$$
	Then $\mathcal{G}(h,L,M,N)$ is M-D-G frame with frame bounds
	$$A\left(1-\displaystyle{\sqrt{\frac{R}{A}}}\right)^2, \; \;  B\left(1+\displaystyle{\sqrt{\frac{R}{B}}}\right)^2.$$
	\end{theorem}
	
\begin{proof}
Let $f\in \ell_0(\mathbb{S})$. Applying lemma \ref{lem4} for $g-h=\{g_l-h_l\}_{l\in \mathbb{N}_L}$, we have: 
$$\begin{array}{rcl}
&&\displaystyle{\sum_{l\in \mathbb{N}_L}\sum_{n\in \mathbb{Z}}\sum_{m\in \mathbb{N}_M}\langle f, E_{\frac{m}{M}}T_{nN}(g_l-h_l)\rangle\vert^2}\\&=&M\displaystyle{\sum_{j\in \mathbb{S}} \left(\sum_{l\in \mathbb{N}_L} \mathcal{M}_{g_l-h_l}(j)\mathcal{M}_{g_l-h_l}^*(j)\right)_{0,0} \, \vert f(j)\vert^2}\\
&+&M\displaystyle{\sum_{k\in \mathbb{Z}-\{0\}}\sum_{j\in \mathbb{S}}\left(\sum_{l\in \mathbb{N}_L}\mathcal{M}_{g_l-h_l}(j)\mathcal{M}_{g_l-h_l}^*(j)\right)_{0,k} f(j)\, \overline {f(j+kM)}}\\
&=&F_1(f)+F_2(f).\\ 
\end{array}$$
	On the other hand, applying Cauchy-Schwarz inequality twice, we have for all $f\in \ell_0(\mathbb{S})$,
	$$\begin{array}{rcl}
	\vert F_2(f)\vert &=& \left\vert M\displaystyle{\sum_{k\in \mathbb{Z}-\{0\}}\sum_{j\in \mathbb{S}}\left(\sum_{l\in \mathbb{N}_L}\mathcal{M}_{g_l-h_l}(j)\mathcal{M}_{g_l-h_l}^*(j)\right)_{0,k} f(j)\, \overline {f(j+kM)}}\right\vert\\

	&\leq& M\displaystyle{\sum_{k\in \mathbb{Z}-\{0\}}\sum_{j\in \mathbb{S}}\left\vert \left(\sum_{l\in \mathbb{N}_L}\mathcal{M}_{g_l-h_l}(j)\mathcal{M}_{g_l-h_l}^*(j)\right)_{0,k}\right\vert.\,\vert f(j)f(j+kM)\vert}\\
	&=& M\displaystyle{\sum_{k\in \mathbb{Z}-\{0\}}\sum_{j\in \mathbb{S}}\left\vert \left(\sum_{l\in \mathbb{N}_L}\mathcal{M}_{g_l-h_l}(j)\mathcal{M}_{g_l-h_l}^*(j)\right)_{0,k}\right\vert^{\frac{1}{2}}\,\vert f(j)\vert}\\
	&\times& \left\vert \left(\sum_{l\in \mathbb{N}_L}\mathcal{M}_{g_l-h_l}(j)\mathcal{M}_{g_l-h_l}^*(j)\right)_{0,k}\right\vert^{\frac{1}{2}}\vert f(j+kM)\vert\\
	&\leq& M\displaystyle{\sum_{k\in \mathbb{Z}-\{0\}}\left(\sum_{j\in \mathbb{S}}\left\vert \left(\sum_{l\in \mathbb{N}_L}\mathcal{M}_{g_l-h_l}(j)\mathcal{M}_{g_l-h_l}^*(j)\right)_{0,k}\right\vert\, \vert f(j)\vert^2\right)^{\frac{1}{2}}}\\
	&\times& \displaystyle{\left(\sum_{j\in \mathbb{S}}\left\vert \left(\sum_{l\in \mathbb{N}_L}\mathcal{M}_{g_l-h_l}(j)\mathcal{M}_{g_l-h_l}^*(j)\right)_{0,k}\right\vert\, \vert f(j+kM)\vert^2\right)^{\frac{1}{2}}}\\
&\leq& 	M\left(\displaystyle{\sum_{k\in \mathbb{Z}-\{0\}}\sum_{j\in \mathbb{S}}\left\vert \left(\sum_{l\in \mathbb{N}_L}\mathcal{M}_{g_l-h_l}(j)\mathcal{M}_{g_l-h_l}^*(j)\right)_{0,k}\right\vert\, \vert f(j)\vert^2}\right)^{\frac{1}{2}}\\
	&\times& \left(\displaystyle{\sum_{k\in \mathbb{Z}-\{0\}}\sum_{j\in \mathbb{S}}\left\vert \left(\sum_{l\in \mathbb{N}_L}\mathcal{M}_{g_l-h_l}(j)\mathcal{M}_{g_l-h_l}^*(j)\right)_{0,k}\right\vert\, \vert f(j+kM)\vert^2}\right)^{\frac{1}{2}}.\\
	\end{array}$$
	
	Since for $k\neq 0$, if $j\notin \mathbb{S}$, then $\left(\mathcal{M}_{g_l-h_l}(j)\mathcal{M}_{g_l-h_l}^*(j)\right)_{0,k}=0$, then applying  the substitution variable $k\mapsto -k$ and  $j+kM\mapsto j$, we have:$$\begin{array}{rcl}
	&&\displaystyle{\sum_{k\in \mathbb{Z}-\{0\}}\sum_{j\in \mathbb{S}}\left\vert \left(\sum_{l\in \mathbb{N}_L}\mathcal{M}_{g_l-h_l}(j)\mathcal{M}_{g_l-h_l}^*(j)\right)_{0,k}\right\vert\, \vert f(j+kM)\vert^2}\\&=&
	\displaystyle{\sum_{k\in \mathbb{Z}-\{0\}}\sum_{j\in \mathbb{Z}}\left\vert \left(\sum_{l\in \mathbb{N}_L}\mathcal{M}_{g_l-h_l}(j)\mathcal{M}_{g_l-h_l}^*(j)\right)_{0,-k}\right\vert\, \vert f(j-kM)\vert^2}\\
&=&	\displaystyle{\sum_{k\in \mathbb{Z}-\{0\}}\sum_{j\in \mathbb{S}}\left\vert \left(\sum_{l\in \mathbb{N}_L}\mathcal{M}_{g_l-h_l}(j+kM)\mathcal{M}_{g_l-h_l}^*(j+kM)\right)_{0,-k}\right\vert\, \vert f(j)\vert^2}\\

&=&	\displaystyle{\sum_{k\in \mathbb{Z}-\{0\}}\sum_{j\in \mathbb{S}}\left\vert \left(\sum_{l\in \mathbb{N}_L}\mathcal{M}_{g_l-h_l}(j)\mathcal{M}_{g_l-h_l}^*(j)\right)_{0,k}\right\vert\, \vert f(j)\vert^2}\\
	&=&\displaystyle{\sum_{j\in \mathbb{S}}\sum_{k\in \mathbb{Z}-\{0\}}\left\vert \left(\sum_{l\in \mathbb{N}_L}\mathcal{M}_{g_l-h_l}(j)\mathcal{M}_{g_l-h_l}^*(j)\right)_{0,k}\right\vert\, \vert f(j)\vert^2}. \;\; \; (\text{ since } f\in \ell_0(\mathbb{S})\,)
	\end{array}$$
	Hence: $$F_2(f)\leq  M\displaystyle{\sum_{k\in \mathbb{Z}-\{0\}}\sum_{j\in \mathbb{S}}\left\vert \left(\sum_{l\in \mathbb{N}_L}\mathcal{M}_{g_l-h_l}(j)\mathcal{M}_{g_l-h_l}^*(j)\right)_{0,k}\right\vert\, \vert f(j)\vert^2}.$$
	Hence: $$\begin{array}{rcl}
		&&\displaystyle{\sum_{l\in \mathbb{N}_L}\sum_{n\in \mathbb{Z}}\sum_{m\in \mathbb{N}_M}\langle f, E_{\frac{m}{M}}T_{nN}(g_l-h_l)\rangle\vert^2}\\
		&\leq& \displaystyle{\sum_{j\in \mathbb{S}}\left(M\sum_{k\in \mathbb{Z}}\left\vert \left(\sum_{l\in \mathbb{N}_L}\mathcal{M}_{g_l-h_l}(j) \mathcal{M}_{g_l-h_l}^*(j)\right)_{0,k}\right\vert\right).\; \vert f(j)\vert^2}\\
	&\leq &R\|f\|^2. \;\; (\text{ By N-periodicity  of } \mathcal{M}_{g_l-h_l}(.)\mathcal{M}_{g_l-h_l}^*(.)\;\;)  
	\end{array}$$
	Since $R< A$ and by density of $\ell_0(\mathbb{S})$ in $\ell^2(\mathbb{S})$,  then lemma \ref{lem3} completes the proof.\\
	\end{proof}

	We finish this section by genzralizing the theorem 3.1 in \cite{10} to M-D-G frames for $\ell^2(\mathbb{Z})$. We show, therefore,  that every pair of  M-D-G Bessel systems can generate  pairs of M-D-G duals frames.
	
	\begin{proposition}
	Let $g=\{g_l\}_{l\in \mathbb{N}_L}, h=\{h_l\}_{l\in \mathbb{N}_L}\subset \ell^2(\mathbb{Z})$ such that $\mathcal{G}(g,L,M,N)$ and $\mathcal{G}(h,L,M,N)$ are two Bessel systems. Let $K\in \mathbb{N}$ be such that $N\leq KM$. Then there exist $\{g_l\}_{l=L,...,L+K-1}, \{h_l\}_{l=L,...,L+K-1}\subset \ell^2(\mathbb{Z})$ such that $\mathcal{G}(\;\{g_l\}_{l\in \mathbb{N}_{L+K}}, L+K,M,N)$ and $\mathcal{G}(\;\{h_l\}_{l\in \mathbb{N}_{L+K}}, L+K,M,N)$ form dual frames in $\ell^2(\mathbb{Z})$.\\
	\end{proposition}
	
\begin{proof}
Denote by $U_g$ and $U_h$ the synthesis operators of $\mathcal{G}(g,L,M,N)$ and $\mathcal{G}(h,L,M,N)$ respectively. Define $\Psi:=I-U_gU_h^*$. Let $g':=\{g'_l\}_{l\in \mathbb{N}_K}, \; h':=\{h'_l\}_{l\in \mathbb{N}_K}\in \ell^2(\mathbb{Z})$ such that $\mathcal{G}(g',K,M,N)$ and $\mathcal{G}(h',K,M,N)$ are two dual D-G frames for $\ell^2(\mathbb{Z})$ ( the existence is proved in Theorem \ref{exframe}). Then we have for $f\in \ell^2(\mathbb{Z})$, $$\begin{array}{rcl}
&&f=\Psi(f)+\displaystyle{\sum_{l\in \mathbb{N}_L}\sum_{n\in \mathbb{Z}}\sum_{m\in \mathbb{N}_M}\langle f,E_{\frac{m}{M}}T_{nN}h_l\rangle E_{\frac{m}{M}}T_{nN}g_l}\\
&=& \displaystyle{\sum_{l\in \mathbb{N}_K}\sum_{n\in \mathbb{Z}}\sum_{m\in \mathbb{N}_M}\langle \Psi(f), E_{\frac{m}{M}}T_{nN}h'_l\rangle E_{\frac{m}{M}}T_{nN}g'_l}+\displaystyle{\sum_{l\in \mathbb{N}_L}\sum_{n\in \mathbb{Z}}\sum_{m\in \mathbb{N}_M}\langle f,E_{\frac{m}{M}}T_{nN}h_l\rangle E_{\frac{m}{M}}T_{nN}g_l}\\
&=&\displaystyle{\sum_{l\in \mathbb{N}_K}\sum_{n\in \mathbb{Z}}\sum_{m\in \mathbb{N}_M}\langle f, \Psi^*E_{\frac{m}{M}}T_{nN}h'_l\rangle E_{\frac{m}{M}}T_{nN}g'_l}+\displaystyle{\sum_{l\in \mathbb{N}_L}\sum_{n\in \mathbb{Z}}\sum_{m\in \mathbb{N}_M}\langle f,E_{\frac{m}{M}}T_{nN}h_l\rangle E_{\frac{m}{M}}T_{nN}g_l}.\\
\end{array}$$
By lemma \ref{lem2}, we have $U_hU_g^*E_{\frac{m}{M}}T_{nN}=E_{\frac{m}{M}}T_{nN} U_hU_g^*$. Then $\Psi^* E_{\frac{m}{M}}T_{nN}=E_{\frac{m}{M}}T_{nN}\Psi^*$. Hence for all $f\in \ell^2(\mathbb{Z})$, we have: $$
f=\displaystyle{\sum_{l\in \mathbb{N}_K}\sum_{n\in \mathbb{Z}}\sum_{m\in \mathbb{N}_M}\langle f, E_{\frac{m}{M}}T_{nN}\Psi^*h'_l\rangle E_{\frac{m}{M}}T_{nN}g'_l}+\displaystyle{\sum_{l\in \mathbb{N}_L}\sum_{n\in \mathbb{Z}}\sum_{m\in \mathbb{N}_M}\langle f,E_{\frac{m}{M}}T_{nN}h_l\rangle E_{\frac{m}{M}}T_{nN}g_l}.
$$
Take for all $l\in \mathbb{N}_K$, $g_{L+l}=g'_l$ and $h_{L+l}=\Psi^*\,h'_l$, then for all $f\in \ell^2(\mathbb{Z})$, we have
$$f=\displaystyle{\sum_{l\in \mathbb{N}_{L+K}}\sum_{n\in \mathbb{Z}}\sum_{m\in \mathbb{N}_M}\langle f,E_{\frac{m}{M}}T_{nN}h_l\rangle E_{\frac{m}{M}}T_{nN}g_l}.$$
Hence $\mathcal{G}(\;\{g_l\}_{l\in \mathbb{N}_{L+K}}, L+1,M,N)$ and $\mathcal{G}(\;\{h_l\}_{l\in \mathbb{N}_{L+K}}, L+1,M,N)$ form dual frames in $\ell^2(\mathbb{Z})$.\\
\end{proof}	

\section{ M-D-G frames for $\ell^2(\mathbb{Z})$ and discrete Zak transform}
In this section, we study M-D-G complete systems and M-D-G frames for $\ell^2(\mathbb{Z})$ in the case of $N=M$ and then  M-D-G bases for $\ell^2(\mathbb{Z})$.\\
	
We start by the case $M=N$. 
Let $g:=\{g_l\}_{l\in \mathbb{N}_L}\subset \ell^2(\mathbb{Z})$ and denote for all $l\in \mathbb{N}_L, m\in \mathbb{N}_M, n\in \mathbb{Z}$,  $\;g_{m,n,l}:=E_{\frac{m}{M}}T_{nN}g_l$. Under the condition $M=N$, we have: $$\begin{array}{rcl}
z_Mg_{m,n,l}(j,\theta)&=&\displaystyle{\sum_{k\in \mathbb{Z}}g_{m,n,l}(j+kM)e^{2\pi i k\theta}}\\
&=& \displaystyle{\sum_{k\in \mathbb{Z}}e^{2\pi i\frac{m}{M}(j+kM)}g_l(j+kM-nM)e^{2\pi i k \theta}}\\
&=& e^{2\pi i \frac{m}{M}j}\, e^{2\pi i n\theta}\, \displaystyle{\sum_{k\in \mathbb{Z}}g_l(j+(k-n)M)e^{2\pi i (k-n)\theta}}\\
&=& E_{(\frac{m}{M},n)}(j,\theta). z_M g_l(j,\theta)\\  \\
&=& (E_{(\frac{m}{M},n)}.z_Mg_l)(j,\theta).\\
\end{array}	$$
What makes this cas of $M=N$ interesting is the fact that $\{\displaystyle{\frac{1}{\sqrt{M}}}E_{(\frac{m}{M},n)}\}_{m\in \mathbb{N}_M, n\in \mathbb{Z}}$ is an orthonormal basis for $\ell^2(q)$.\\

The following result characterizes  complete M-D-G systems  for $\ell^2(\mathbb{Z})$ whenever $N=M$.
	\begin{theorem}\label{nm}
Assume that $N=M$ and	let $g:=\{g_l\}_{l\in \mathbb{N}_L}\subset \ell^2(\mathbb{Z})$ and $\mathcal{G}(g,L,M,N)$ be the associated M-D-G system. Then following statements are equivalent:
	\begin{enumerate}
	\item $\mathcal{G}(g,L,M,N)$ is complete in $\ell^2(\mathbb{Z})$.
	\item $\forall j\in \mathbb{N}_M,\; mes\left( \displaystyle{\bigcap_{l\in \mathbb{N}_L}Z^q_j(g_l)}\right)=0$,
	where $Z^q_j(g_l):=\{\theta\in  [0,1[:\; z_Mg_l(j,\theta)=0\}$.
	\item $\forall j\in \mathbb{Z},\; mes\left( \displaystyle{\bigcap_{l\in \mathbb{N}_L}Z_j(g_l)}\right)=0$,
	where $Z_j(g_l):=\{\theta\in  \mathbb{R}:\; z_Mg_l(j,\theta)=0\}$.\\
	\end{enumerate}
	\end{theorem}
	\begin{proof}\hspace{1cm}
	\begin{enumerate}
	\item[]$(1)\Longrightarrow(2):$  Assume that $\mathcal{G}(g,L,M,N)$ is complete in $\ell^2(\mathbb{Z})$, then $\{E_{(\frac{m}{M},n)}z_Mg_l\}_{m\in \mathbb{N}_M, n\in \mathbb{Z}, l\in \mathbb{N}_L}$ is complete in $\ell^2(q)$ by the unitarity of  $z_M$. Fix $j_0\in \mathbb{N}_M$ and  define $\phi \in \ell^2(q)$ as follows: $$\phi(j,\theta):=\left\lbrace
	\begin{array}{rcl}
	&1&\;\;\; \text{ if } j=j_0 \text{ and } \forall l\in \mathbb{N}_L,\; z_Mg_l(j,\theta)=0\\
	&0& \;\;\; \text{ otherwise}.
	\end{array}
	\right.$$
	We have for all $l\in \mathbb{N}_L$, $n\in \mathbb{Z}$ and $m\in \mathbb{N}_M$,
	 $$\begin{array}{rcl}
	\langle z_Mg_{m,n,l},\phi\rangle&=&\langle E_{(\frac{m}{M},n)}.z_Mg_l,\phi\rangle\\
	&=& \langle E_{(\frac{m}{M},n)},\phi. \overline{z_Mg_l}\rangle\;\; (\text{ Since }\phi \text{ is bounded, then }\phi. \overline{z_Mg_l}\in \ell^2(q)\;)\\
	&=& 0.
	\end{array}$$
Then $\phi=0$ a.e. Thus, by definition of $\phi$, we have $mes\left(\{j_0\}\times \displaystyle{\bigcap_{l\in \mathbb{N}_L}Z^q_{j_0}(g_l)}\right)=0$. Hence $mes\left(\displaystyle{\bigcap_{l\in \mathbb{N}_L}Z^q_{j_0}(g_l)}\right)=0$ (for all $j_0\in \mathbb{N}_M$). \\
\item[]$(2)\Longrightarrow(1)$: Assume that for all $j\in \mathbb{N}_M$, $mes\left(\displaystyle{\bigcap_{l\in \mathbb{N}_L}Z^q_j(g_l)}\right)=0$. Let $\phi\in \ell^2(q)$ such that for all $l\in \mathbb{N}_L$, $n\in \mathbb{Z}$ and $m\in \mathbb{N}_M$, $\langle z_Mg_{m,n,l},\phi\rangle=0$, i.e. $\langle E_{(\frac{m}{M},n)}.z_Mg_l,\phi\rangle=0$. Then $\langle E_{(\frac{m}{M},n)}, \phi .\overline{z_Mg_l}\rangle=0$, thus $\phi .\overline{z_Mg_l}=0$ since $\{E_{(\frac{m}{M},n)}\}_{m\in \mathbb{N}_M, n\in \mathbb{Z}}$ is complete also in $\ell^1(q)$. Hence for all $j\in \mathbb{N}_M$, $\phi(j,.)=0$ a.e. Thus $\|\phi\|^2=\displaystyle{\sum_{j\in \mathbb{N}_M}\int_0^1\vert \phi(j,\theta)\vert^2}=0$, then $\phi=0$ a.e. Thus $\{z_Mg_{m,n,l}\}$ is complete in $\ell^2(q)$. Hence $\mathcal{G}(g,L,M,N)$ is complete in $\ell^2(\mathbb{Z})$ by the unitarity of $z_M$. \\
\item[]$(2)\Longrightarrow(3):$ Let $j\in \mathbb{N}_M$, let's show that $\displaystyle{\bigcap_{l\in \mathbb{N}_L}Z_jg_l=\bigcup_{p\in \mathbb{Z}}\left(\bigcap_{l\in \mathbb{N}_L}Z_j^q(g_l)+p\right)}$.\\
We have for all $\theta\in \mathbb{R}$, $p\in \mathbb{Z}$, $z_Mg_l(j,\theta+p)=z_Mg_l(j,\theta)$, then:
$$\begin{array}{rcl}
\theta \in \displaystyle{\bigcap_{l\in \mathbb{N}_L}Z_j(g_l)}&\Longrightarrow& \forall l\in \mathbb{N}_L,\;  z_Mg_l(j,\theta)=0\\
&\Longrightarrow& \forall l\in \mathbb{N}_L,\;  z_Mg_l(j,\theta_0+p)=0\;\;\; \text{ where }\theta=\theta_0+p \text{ and } \theta_0\in [0,1[\\
&\Longrightarrow& \forall l\in \mathbb{N}_L,\;  z_Mg_l(j,\theta_0)=0\\
&\Longrightarrow&\forall l\in \mathbb{N}_L,\; \theta_0\in Z^q_j(g_l)\\
&\Longrightarrow& \theta_0\in \displaystyle{\bigcap_{l\in \mathbb{N}_L}Z^q_j(g_l)}\\
&\Longrightarrow& \theta \in \displaystyle{\bigcap_{l\in \mathbb{N}_L}Z^q_j(g_l)}+p\\
&\Longrightarrow& \theta \in \displaystyle{\bigcup_{p\in \mathbb{Z}}\left(\bigcap_{l\in \mathbb{N}_L}Z^q_j(g_l)+p\right)}.
\end{array}$$
Conversely, let $\theta \in \displaystyle{\bigcup_{p\in \mathbb{Z}}\left(\bigcap_{l\in \mathbb{N}_L}Z^q_j(g_l)+p\right)}$, then there exist $p\in \mathbb{Z}$, $\theta_0\in \displaystyle{\bigcap_{l\in \mathbb{N}_L}Z^q_j(g_l)}$ such that $\theta =\theta_0+p$. Since $\forall l\in \mathbb{N}_M$, $z_Mg_l(j,\theta_0)=0$, then $\forall l\in \mathbb{N}_M$, $z_Mg_l(j,\theta_0+p)=0$. Hence $\theta \in \displaystyle{\bigcap_{l\in \mathbb{N}_L}Z_j(g_l)}$. \\
We have for all $j,k\in \mathbb{Z}$, $\theta\in \mathbb{R}$, $z_Mg_l(j,\theta)=e^{-2\pi i k\theta}z_Mg_l(j+kM,\theta)$. Let $j\in \mathbb{Z}$ and let $j_0\in \mathbb{N}_M$, $k\in \mathbb{Z}$ such that $j=j_0+kM$. Then for all $l\in \mathbb{N}_L$, $Z_j(g_l)=Z_{j_0}(g_l)$. Hence:
$$\displaystyle{\bigcap_{l\in \mathbb{N}_L}Z_j(g_l)= \displaystyle{\bigcap_{l\in \mathbb{N}_L}Z_{j_0}(g_l)= \bigcup_{p\in \mathbb{Z}}\left(\bigcap_{l\in \mathbb{N}_L}Z_{j_0}^q(g_l)+p\right)}}.$$
Since Lebsegue measure is invariant by translations and that a countable union of measure zero sets is also a measure zero set, then we obtain the equivalence $(2)\Longleftrightarrow (3)$.\\
\end{enumerate}
\end{proof}

	The following result characterizes   M-D-G frames for $\ell^2(\mathbb{Z})$ whenever $N=M$.
\begin{theorem}\label{delta}
Assume that $N=M$ and Let $\mathcal{G}(g,L,M,N)$ be a M-D-G system in $\ell^2(\mathbb{Z})$. Then the following statements are equivalent:
	\begin{enumerate}
	\item $\mathcal{G}(g,L,M,N)$ is a frame for $\ell^2(\mathbb{Z})$ with frame bounds $A$ and $B$.
	 \item  $\forall j\in \mathbb{N}_M,\;\displaystyle{\frac{A}{M}}\leq \displaystyle{\sum_{l\in \mathbb{N}_L}\vert z_Mg_l(j,.)\vert^2}\leq \displaystyle{\frac{B}{M}} \; \,\text{ a.e on }[0,1[.$
	\item  $\forall j\in \mathbb{Z},\;\displaystyle{\frac{A}{M}}\leq \displaystyle{\sum_{l\in \mathbb{N}_L}\vert z_Mg_l(j,.)\vert^2}\leq \displaystyle{\frac{B}{M}} \; \,\text{ a.e on }\mathbb{R}.$\\
\end{enumerate}
\end{theorem}
	
	\begin{proof}\hspace{1cm}
	\begin{enumerate}
	\item[]$(1)\Longrightarrow (2)$: Assume that $\mathcal{G}(g,L,M,N)$ is a frame for $\ell^2(\mathbb{Z})$ with frame bounds $A$ and $B$. Then $\{z_Mg_{m,n,l}\}$ is a frame for $\ell^2(q)$ with frame bounds $A$ and $B$ by the unitarity of $z_M$. Then for all $F\in \ell^2(q)$ bounded, we have: $$A\|F\|^2\leq \displaystyle{\sum_{l\in \mathbb{N}_L}\sum_{n\in \mathbb{Z}}\sum_{m\in \mathbb{N}_M}\vert \langle F, E_{(\frac{m}{M},n)}.z_Mg_l\rangle\vert^2}\leq B\|F\|^2.$$
We have $F.\overline{z_Mg_l}\in \ell^2(q)$ since $F$ is bounded. Then: 
$$A\|F\|^2\leq \displaystyle{\sum_{l\in \mathbb{N}_L}\sum_{n\in \mathbb{Z}}\sum_{m\in \mathbb{N}_M}\vert \langle F.\overline{z_Mg_l}, E_{(\frac{m}{M},n)}\rangle\vert^2}\leq B\|F\|^2.$$
Thus: 
$$A\|F\|^2\leq M\displaystyle{\sum_{l\in \mathbb{N}_L}\|  F.\overline{z_Mg_l}\|^2}\leq B\|F\|^2.$$
Then: $$A\displaystyle{\sum_{j\in \mathbb{N}_M}\int_{0}^1\vert F(j,\theta)\vert^2\, d\theta}\leq M\displaystyle{\sum_{l\in \mathbb{N}_L}\sum_{j\in \mathbb{N}_M}\int_0^1 \vert F(j,\theta).z_Mg_l(j,\theta)\vert^2}\leq B\displaystyle{\sum_{j\in \mathbb{N}_M}\int_{0}^1\vert F(j,\theta)\vert^2\, d\theta}.$$
Fix $j_0\in \mathbb{N}_M$ and $f\in L^2([0,1[)$ bounded. Define for all $(j,\theta)\in q$, $F_0(j,\theta):=\delta_{j,j_0}.\, f(\theta)$. Applying the above inequality on $F_0$, we obtain for all $f\in L^2([0,1[)$ bounded:
\begin{equation}
A\displaystyle{\int_0^1\vert f(\theta)\vert^2\, d\theta}\leq M\displaystyle{\int_{0}^1 \vert f(\theta)\vert^2.\sum_{l\in \mathbb{N}_L}\vert \, z_Mg_l(j_0,\theta)\vert^2 \, d\theta}\leq B\displaystyle{\int_0^1\vert f(\theta)\vert^2\, d\theta}.
\end{equation}
Let's show, now, that $A\leq M \displaystyle{\sum_{l\in \mathbb{N}_L}\vert z_Mg_l(j_0,\theta)\vert^2}$ a.e. Suppose , by contradiction, that $mes(\Delta)>0$, where $\Delta:=\{\theta \in [0,1[:\; A>M \displaystyle{\sum_{l\in \mathbb{N}_L}\vert z_Mg_l(j_0,\theta)\vert^2}\}$. For all $k\in \mathbb{N}$, define: $$\Delta_k:=\{\theta \in [0,1[:\;    A- \displaystyle{\frac{A}{k}}\leq M\displaystyle{\sum_{l\in \mathbb{N}_L}\vert z_Mg_l(j_0,\theta)\vert^2}\leq A- \displaystyle{\frac{A}{k+1}}\}.$$
We have, then, $\displaystyle{\bigcup_{k\in \mathbb{N}}\Delta_k}=\Delta$. Since $mes(\Delta)>0$, then there exists $k_0\in \mathbb{N}$ such that $mes(\Delta_{k_0})> 0$. Take  $f_0=\chi_{\Delta_{k_0}}$, then, by (4.1), we have: $$\begin{array}{rcl}
M\displaystyle{\int_0^1 \sum_{l\in \mathbb{N}_L}\vert f_0(\theta).z_Mg_l(j,\theta)\vert^2}&=&M\displaystyle{\int_{\Delta_{k_0}} \sum_{l\in \mathbb{N}_L}\vert z_Mg_l(j,\theta)\vert^2}\\
&\leq& \left(A-\displaystyle{\frac{A}{k+1}}\right). mes(\Delta_{k_0})\\
&=&  \left(A-\displaystyle{\frac{A}{k+1}}\right). \|f_0\|^2\\
&<& A\|f_0\|^2.
\end{array}$$
Contradiction with (1). Then for all $j_0\in \mathbb{N}_M$, $A\leq M \displaystyle{\sum_{l\in \mathbb{N}_L}\vert z_Mg_l(j_0,\theta)\vert^2}$ a.e.\\
Let's show, also, that $ M \displaystyle{\sum_{l\in \mathbb{N}_L}\vert z_Mg_l(j_0,\theta)\vert^2}\leq B$ a.e. Suppose, by contradiction again, that $mes(D)>0$, where $D:=\{\theta \in [0,1[:\; B<M \displaystyle{\sum_{l\in \mathbb{N}_L}\vert z_Mg_l(j_0,\theta)\vert^2}\}$. For all $k\in \mathbb{N}, m\in \mathbb{N}$, define: $$
D_{k,m}:=\{\theta \in [0,1[:\; \displaystyle{B(k+\frac{1}{m+1})}\leq M \displaystyle{\sum_{l\in \mathbb{N}_L}\vert z_Mg_l(j_0,\theta)\vert^2}\}\leq \displaystyle{B(k+\frac{1}{m})}\}.$$
We have, then, $\displaystyle{\bigcup_{k\in \mathbb{N}_0,m\in \mathbb{N}}D_{k,m}}=D$. Since $mes(D)>0$, then there exists $k_0\in \mathbb{N}$ and $m_0\in \mathbb{N}$ such that $mes(D_{k_0,m_0})> 0$. Take  $f_0=\chi_{D_{k_0,m_0}}$, then we have: $$\begin{array}{rcl}
M\displaystyle{\int_0^1 \sum_{l\in \mathbb{N}_L}\vert f_0(\theta).z_Mg_l(j,\theta)\vert^2}&=&M\displaystyle{\int_{D_{k_0,m_0}} \sum_{l\in \mathbb{N}_L}\vert z_Mg_l(j,\theta)\vert^2}\\
&\geqslant & B\left(k_0+\displaystyle{\frac{1}{m_0}}\right). mes(D_{k_0,m_0})\\
&=&  B\left(k_0+\displaystyle{\frac{1}{m_0}}\right). \|f_0\|^2\\
&>& B\|f_0\|^2.
\end{array}$$
Contradiction with (1). Then for all $j_0\in \mathbb{N}_M$, $ M \displaystyle{\sum_{l\in \mathbb{N}_L}\vert z_Mg_l(j_0,\theta)\vert^2}\leq B$ a.e.\\
Hence: $$\displaystyle{\frac{A}{M}}\leq \displaystyle{\sum_{l\in \mathbb{N}_L}\vert z_Mg_l(.,.)\vert^2}\leq \displaystyle{\frac{B}{M}} \; \,\text{ a.e}.$$
\item[$2\Longrightarrow1)$] Assume that for all $j\in \mathbb{N}_M$, $\displaystyle{\frac{A}{M}}\leq \displaystyle{\sum_{l\in \mathbb{N}_L}\vert z_Mg_l(j,.)\vert^2}\leq \displaystyle{\frac{B}{M}} \; \,\text{ a.e}.$ Since for all $l\in \mathbb{N}_L$,  $z_Mg_l$ is bounded then for all $F\in \ell^2(q)$, $F.\overline{z_Mg_l}\in \ell^2(q)$. Then we have 
$$\begin{array}{rcl}
\displaystyle{\sum_{l\in \mathbb{N}_L}\sum_{n\in \mathbb{Z}}\sum_{m\in \mathbb{N}_M}\vert \langle F,E_{(\frac{m}{M},n)}.z_Mg_l\rangle \vert^2}&=&\displaystyle{\sum_{l\in \mathbb{N}_L}\sum_{n\in \mathbb{Z}}\sum_{m\in \mathbb{N}_M}\vert \langle F.\overline{z_Mg_l},E_{(\frac{m}{M},n)}\rangle \vert^2}\\
&=& M\displaystyle{\sum_{l\in \mathbb{N}_L}\| F.\overline{z_Mg_l}\|^2}\\
&=& M\displaystyle{\sum_{l\in \mathbb{N}_L}\sum_{j\in \mathbb{N}_M}\int_0^1\vert F(j,\theta).z_Mg_l(j,\theta)\vert^2}\\
&=& \displaystyle{\sum_{j\in \mathbb{N}_M}\int_0^1\,\left(M\sum_{l\in \mathbb{N}_L}\vert z_Mg_l(j,\theta)\vert^2\right). \, \vert F(j,\theta)\vert^2}.\\
\end{array}$$
Then for all $F\in \ell^2(q)$, we have: 
$$A\|F\|^2\leq \displaystyle{\sum_{l\in \mathbb{N}_L}\sum_{n\in \mathbb{Z}}\sum_{m\in \mathbb{N}_M}\vert \langle F,E_{(\frac{m}{M},n)}.z_Mg_l\rangle \vert^2}\leq B\|F\|^2.$$
Thus $\{E_{(\frac{m}{M},n)}.z_Mg_l\}$ is a frame for $\ell^2(q)$ with frame bounds $A$ and $B$. Hence $\mathcal{G}(g,L,M,N)$ is a frame for $\ell^2(\mathbb{Z})$ with frame bounds $A$ and $B$ by the uinitarity of $z_M$.\\ 
\item[]$(2)\Longleftrightarrow(3):$ Since for all $j,k,p\in \mathbb{Z}$, $z_Mg_l(j+kM,\theta+p)=e^{-2\pi i k\theta}z_Mg_l(j,\theta)$, then for all $j,k,p\in \mathbb{Z}$, $\displaystyle{\sum_{l\in \mathbb{N}_L}\vert z_Mg_l(j+kM,\theta+p)\vert^2} =\displaystyle{\sum_{l\in \mathbb{N}_L}\vert z_Mg_l(j,\theta)\vert^2}$.\\
Denote $A_j:=\{\theta \in \mathbb{R}: \text{ the inequality in (2) does not hold}\,\}$ and $A^q_j:=\{\theta \in [0,1[: \text{ the inequality in (2) does not hold}\,\}$. We show, easily, that for all $j\in \mathbb{N}_M$, $A_j=\displaystyle{\bigcup_{p\in \mathbb{Z}}A^q_j+p}$. For $j\in \mathbb{Z}$, let $j_0\in \mathbb{N}_M$ and $k\in \mathbb{Z}$ such that $j=j_0+kM$, then $A_j=\displaystyle{\bigcup_{p\in \mathbb{Z}}A^q_{j_0}+p}$. Since Lebesgue measure is invariant by translations and that a countable union of measure zero sets is a measure zero set, then we deduce the equivalence $(2)\Longleftrightarrow (3)$. \\
\end{enumerate}
	\end{proof}

	\begin{remark}\label{caseL1}
	If $M=N$ and $L=1$, then:
	\begin{enumerate}
	\item $\;\; \mathcal{G}(g,L,M,N)$ is a Riesz basis for $\ell^2(\mathbb{Z})$ with frame bounds $A$ and $B$.\\$\Longleftrightarrow \mathcal{G}(g,L,M,N)$ is a frame for $\ell^2(\mathbb{Z})$ with frame bounds $A$ and $B$.\\
	 $\Longleftrightarrow$ $\forall j\in \mathbb{N}_M,\;  \displaystyle{\frac{A}{M}}\leq \displaystyle{\vert zg(j,.)\vert^2}\leq \displaystyle{\frac{B}{M}} \; \,\text{ a.e on }[0,1[.$\\
	  $\Longleftrightarrow$ $\forall j\in \mathbb{Z},\;  \displaystyle{\frac{A}{M}}\leq \displaystyle{\vert zg(j,.)\vert^2}\leq \displaystyle{\frac{B}{M}} \; \,\text{ a.e on }\mathbb{R}.$\\

	  \item $ \;\;\mathcal{G}(g,L,M,N)$ is an orthonormal basis for $\ell^2(\mathbb{Z})$. \\
	  $\Longleftrightarrow$ $\forall j\in \mathbb{N}_M,\; \displaystyle{\vert zg(j,.)\vert^2}=\frac{1}{M}$ a.e on $[0,1[$.\\
	  $\Longleftrightarrow$ $\forall j\in \mathbb{Z},\; \displaystyle{\vert zg(j,.)\vert^2}=\frac{1}{M}$ a.e on $\mathbb{R}$.\\
\end{enumerate}	
	\end{remark}
	The following result characterizes M-D-G frames for $\ell^2(\mathbb{Z})$ in the case of $N=M$ and  that $z_Mg_l(j,.)$ are continuous.
	\begin{proposition}\label{cont}
Assume that $N=M$. Let $g=\{g_l\}_{l\in \mathbb{N}_L}\subset \ell^2(\mathbb{Z})$ such that for all $j\in \mathbb{N}_M$, $l\in \mathbb{N}_L$, $z_Mg_l(j,.)$ is continuous on $\mathbb{R}$. Then the following statements are equivalent:
\begin{enumerate}
\item  $\mathcal{G}(g,L,M,N)$ is a M-D-G frame for $\ell^2(\mathbb{Z})$.
\item $\displaystyle{\bigcap_{l\in \mathbb{N}_L}Z^q(g_l)}=\emptyset$, where $Z^q(g_l):=\{(j,\theta)\in \mathbb{N}_M\times [0,1[:\; z_Mg_l(j,\theta)=0\}$.\\
\item $\displaystyle{\bigcap_{l\in \mathbb{N}_L}Z(g_l)}=\emptyset$, where $Z(g_l):=\{(j,\theta)\in \mathbb{Z}\times \mathbb{R}:\; z_Mg_l(j,\theta)=0\}$.\\
\end{enumerate}
\end{proposition}
\begin{proof}\hspace{1cm}
\begin{enumerate}
\item[]$(1)\Longrightarrow(2):$ Assume that $\mathcal{G}(g,L,M,N)$ is a M-D-G frame for $\ell^2(\mathbb{Z})$ with frame bounds $A$ and $B$.
 Then, by theorem  \ref{delta}, for all $j\in \mathbb{N}_M$, $\displaystyle{\frac{A}{M}}\leq \displaystyle{\sum_{l\in \mathbb{N}_L}\vert z_Mg_l(j,.)\vert^2}\leq \displaystyle{\frac{B}{M}} \; \,\text{ a.e on }[0,1[.$ Then by the continuity of $\displaystyle{\sum_{l\in \mathbb{N}_L}\vert z_Mg_l(.,.)\vert^2}$, for all $(j,\theta)\in q$, $$\displaystyle{\frac{A}{M}}\leq \displaystyle{\sum_{l\in \mathbb{N}_L}\vert z_Mg_l(j,\theta)\vert^2}\leq \displaystyle{\frac{B}{M}}.$$ 
In fact, suppose, by contradiction, that there exists $(j,\theta)\in q$ such that \\$\displaystyle{\frac{A}{M}}> \displaystyle{\sum_{l\in \mathbb{N}_L}\vert z_Mg_l(j,\theta)\vert^2}$ or $\displaystyle{\frac{B}{M}}< \displaystyle{\sum_{l\in \mathbb{N}_L}\vert z_Mg_l(j,\theta)\vert^2}$. We can suppose the case $\displaystyle{\frac{A}{M}}> \displaystyle{\sum_{l\in \mathbb{N}_L}\vert z_Mg_l(j,\theta)\vert^2}$ and the other one can be obtained similarily. by continuity of $\displaystyle{\sum_{l\in \mathbb{N}_L}\vert z_Mg_l(j,.)\vert^2}$, we have $\left(\displaystyle{\sum_{l\in \mathbb{N}_L}\vert z_Mg_l(j,.)\vert^2}\right)^{-1}\left(]-\infty,\displaystyle{\frac{A}{M}}[\,\right)$ is an open set of $\mathbb{R}$, then contains an open interval of $\mathbb{R}$. Hence \\$mes( \left(\displaystyle{\sum_{l\in \mathbb{N}_L}\vert z_Mg_l(j,.)\vert^2}\right)^{-1}\left(]-\infty,\displaystyle{\frac{A}{M}}[\,\right)>0$. Contradiction.\\
Hence for all $(j,\theta)\in q$, $\displaystyle{\frac{A}{M}}\leq \displaystyle{\sum_{l\in \mathbb{N}_L}\vert z_Mg_l(j,\theta)\vert^2}\leq \displaystyle{\frac{B}{M}}.$\\ 
Since $\displaystyle{\frac{A}{M}}>0$, then $\displaystyle{\bigcap_{l\in \mathbb{N}_L}Z^q(g_l)}=\emptyset$.\\
\item[]$(2)\Longrightarrow(1):$ Conversely, assume that $\displaystyle{\bigcap_{l\in \mathbb{N}_L}Z^q(g_l)}=\emptyset$. Since for all $j\in \mathbb{N}_M$, $\displaystyle{\sum_{l\in \mathbb{N}_L}\vert z_Mg_l(j,.)\vert^2}$ is continuous on $\mathbb{R}$, then, by Heine thoerem, we have $C\leq \displaystyle{\sum_{l\in \mathbb{N}_L}\vert z_Mg_l(j,\theta)\vert^2}\leq D$ for all $\theta\in [0,1]$, where $0<C:=\displaystyle{\min_{\theta\in [0,1]}\displaystyle{\sum_{l\in \mathbb{N}_L}\vert z_Mg_l(j,.)\vert^2}}$ and $0<D:=\displaystyle{\max_{\theta\in [0,1]}\displaystyle{\sum_{l\in \mathbb{N}_L}\vert z_Mg_l(j,.)\vert^2}}$. Then, Theorem \ref{delta} finishes the proof.\\
\item[]$(2)\Longleftrightarrow (3):$ Since for all $j,k,p\in \mathbb{Z}$, $z_Mg_l(j+kM,\theta+p)=e^{-2\pi i k\theta}z_Mg_l(j,\theta)$, then $$\{(j,\theta)\in \mathbb{N}_M\times [0,1[:\; z_Mg_l(j,\theta)=0\}=\emptyset \Longleftrightarrow\{(j,\theta)\in \mathbb{Z}\times \mathbb{R}\; z_Mg_l(j,\theta)=0\}=\emptyset.$$
\end{enumerate}
\end{proof}
The following result is an imediat result of the above proposition and lemma $\ref{lem7}$. 
\begin{corollary}
Assume that $N=M$. Let $g=\{g_l\}_{l\in \mathbb{N}_L}\subset \ell^1(\mathbb{Z})$. Then the following statements are equivalent:
\begin{enumerate}
\item  $\mathcal{G}(g,L,M,N)$ is a M-D-G frame for $\ell^2(\mathbb{Z})$.
\item $\displaystyle{\bigcap_{l\in \mathbb{N}_L}Z^q(g_l)}=\emptyset$, where $Z^q(g_l):=\{(j,\theta)\in \mathbb{N}_M\times [0,1[:\; z_Mg_l(j,\theta)=0\}$.\\
\item $\displaystyle{\bigcap_{l\in \mathbb{N}_L}Z(g_l)}=\emptyset$, where $Z(g_l):=\{(j,\theta)\in \mathbb{Z}\times \mathbb{R}:\; z_Mg_l(j,\theta)=0\}$.\\
\end{enumerate}
\end{corollary}
The following result is a direct consequence of the above corollary and lemma $\ref{lem6}$.
\begin{corollary}
Let $M,L\in\mathbb{N}$. Let $g:=\{g_l\}_{l\in \mathbb{N}_L}\subset \ell^1(\mathbb{Z})$ .
\begin{enumerate}
\item If $L=1$ and the window sequence is odd, then $\mathcal{G}(g,M,M)$ is not a D-G frame for $\ell^2(\mathbb{Z})$.\\
\item If there exist $l\neq l'$ such that $g_l$ and $g_{l'}$ are odd, then $\mathcal{G}(g,L,M,M)$ is not a M-D-G frame for $\ell^2(\mathbb{Z})$.\\
\item If $L=1$, $M$ is even and the window sequence is even, then $\mathcal{G}(g,M,M)$ is not a D-G frame for $\ell^2(\mathbb{Z})$.\\
\item Assume that $M$ is even. If there exist $l\neq l'$ such that $g_l$ and $g_{l'}$ are even, then $\mathcal{G}(g,L,M,M)$ is not a M-D-G frame for $\ell^2(\mathbb{Z})$.\\
\end{enumerate}
\end{corollary}   	
	\begin{example}
	The Gaussian sequence is defined by $\{e^{-r k^2}\}_{k\in \mathbb{Z}}\subset \ell^1(\mathbb{Z})$, where $r>0$. \\

	\begin{enumerate}
	\item	Let $L\in \mathbb{N}$. Take $\{g_l\}_{l\in \mathbb{N}_L}\subset \ell^2(\mathbb{Z})$ such that for $l\in \mathbb{N}_L$, $g_l$ is a Gaussian sequence. Let $M$ be an even positive integer. Then $\mathcal{G}(g,l,M,M)$ is not a M-D-G frame.\\
	\item Let $L\in \mathbb{N}$. Take $\{g_l\}_{l\in \mathbb{N}_L}\subset \ell^2(\mathbb{Z})$ such that for $l\in \mathbb{N}_L$, $g_l$ is a Gaussian sequence. Let $M$ be an odd positive integer. Since $z_Mg_l$ has no zero \cite{4}, then $\mathcal{G}(g,l,M,M)$ is  a M-D-G frame for $\ell^2(\mathbb{Z})$ and is a Riesz basis if and only if $L=1$.\\
	\end{enumerate}
	\end{example}
	Now, assume that $N=LM$. We have already shown  in this case the equivalence between a frame and a Riesz basis and then between a Parseval frame and an orthonormal basis. Let $g:=\{g_l\}_{l\in \mathbb{N}_L}\subset \ell^2(\mathbb{Z})$ and denote for all $l\in \mathbb{N}_L, m\in \mathbb{N}_M, n\in \mathbb{Z}$,  $\;g_{m,n,l}:=E_{\frac{m}{M}}T_{nN}g_l$. Under the condition $N=LM$, we have: $$\begin{array}{rcl}
z_Mg_{m,n,l}(j,\theta)&=&\displaystyle{\sum_{k\in \mathbb{Z}}g_{m,n,l}(j+kM)e^{2\pi i k\theta}}\\
&=& \displaystyle{\sum_{k\in \mathbb{Z}}e^{2\pi i\frac{m}{M}(j+kM)}g_l(j+kM-nLM)e^{2\pi i k \theta}}\\
&=& e^{2\pi i \frac{m}{M}j}\, e^{2\pi i nL\theta}\, \displaystyle{\sum_{k\in \mathbb{Z}}g_l(j+(k-nL)M)e^{2\pi i (k-nL)\theta}}\\
&=& E_{(\frac{m}{M},nL)}(j,\theta). z_M g_l(j,\theta)\\  \\
&=& (E_{(\frac{m}{M},nL)}.z_Mg_l)(j,\theta).\\
\end{array}	$$
Using the fact that $\{\displaystyle{\frac{1}{\sqrt{M}}e^{2\pi i \frac{m}{M}.}}\}_{m\in \mathbb{N}_M}$ is an orthonormal basis for $\ell^2(\mathbb{N}_M)$ and $\{\displaystyle{\sqrt{L} \,e^{2\pi i nL.}}\}_{n\in  \mathbb{Z}}$ is an orthonormal basis for $\ell^2([\frac{k}{L},\frac{(k+1)}{L}])$ ( $\forall k\in \mathbb{N}_L$), we give the following necessary conditions, by the Zak tranform, for a M-D-G system to be a Riesz basis and to be an orthonormal basis for $\ell^2(\mathbb{Z})$.
	\begin{theorem}\label{necessary}  Let $\mathcal{G}(g,L,M,N)$ be a M-D-G system such that $N=LM$.
	\begin{enumerate}
	
  \item Assume that $\mathcal{G}(g,L,M,N)$ is a M-D-G Riesz basis for $\ell^2(\mathbb{Z})$ with frame bounds $A$ and $B$. Then: 
	$$ \forall j\in \mathbb{Z},\;\displaystyle{\frac{LA}{M}}\leq \displaystyle{\sum_{l\in \mathbb{N}_L}\vert z_Mg_l(j,.)\vert^2}\leq \displaystyle{\frac{LB}{M}}\; \; \text{ a.e. on }\mathbb{R}.$$
	\item Assume that $\mathcal{G}(g,L,M,N)$ is a M-D-G orthonormal basis for $\ell^2(\mathbb{Z})$. Then:
	$$\forall j\in \mathbb{Z},\;  \displaystyle{\sum_{l\in \mathbb{N}_L}\vert z_Mg_l(j,.)\vert^2}= \displaystyle{\frac{L}{M}}\; \; \text{ a.e. on }\mathbb{R}.$$
	\end{enumerate}	
	\end{theorem}
	\begin{proof}\hspace{1cm}
	\begin{enumerate}
	\item Assume that $\mathcal{G}(g,L,M,N)$ is a M-D-G frame  for $\ell^2(\mathbb{Z})$ with frame bounds $A$ and $B$. Then $\{z_Mg_{m,n,l}\}$ is a frame for $\ell^2(q)$ with frame bounds $A$ and $B$ by the unitarity of $z_M$. Then for all $F\in \ell^2(q)$, we have: 
	$$A\|F\|^2\leq \displaystyle{\sum_{l\in \mathbb{N}_L}\sum_{n\in \mathbb{Z}}\sum_{m\in \mathbb{N}_M}\vert \langle F,E_{(\frac{m}{M},nL)}z_Mg_l\rangle\vert^2}\leq B\| F\|^2.$$
	Fix $j_0\in \mathbb{N}_M$ and $k\in \mathbb{N}_L$. Let $f\in \ell^2([\frac{k}{L},\frac{(k+1)}{L}])$ and define $F_{j_{0},k}(j,\theta):=\delta_{(j,j_0)}. f(\theta)$ for all $(j,\theta)\in q$. Applying the above inequality on $F_{j_{0},k}$, we obtain, after a simple calculation: 
	$$A\int_{\frac{k}{L}}^{\frac{k+1}{L}} \vert f(\theta)\vert^2\, d\theta\leq \frac{M}{L}\int_{\frac{k}{L}}^{\frac{k+1}{L}}\vert f(\theta)\vert^2.\, \sum_{l\in \mathbb{N}_L}\vert z_Mg_l(j_0,\theta)\vert^2\, d\theta\leq B \int_{\frac{k}{L}}^{\frac{k+1}{L}} \vert f(\theta)\vert^2\, d\theta.$$
Using similar arguments used in the proof of Theorem \ref{delta}, we obtain that for all $j\in \mathbb{N}_M$, $k\in \mathbb{N}_L$, $$	
	\frac{LA}{M}\leq \displaystyle{\sum_{l\in \mathbb{N}_L}\vert z_Mg_l(j,.)\vert^2}\leq \displaystyle{\frac{LB}{M}}\; \; \text{ a.e. on } [\frac{k}{L},\frac{k+1}{L}[.$$
	Hence: \begin{equation}
	 \forall j\in \mathbb{N}_M,\;\displaystyle{\frac{LA}{M}}\leq \displaystyle{\sum_{l\in \mathbb{N}_L}\vert z_Mg_l(j,.)\vert^2}\leq \displaystyle{\frac{LB}{M}}\; \; \text{ a.e. on }[0,1[.
	 \end{equation}	
	  Since for all $j,k,p\in \mathbb{Z}$, $z_Mg_l(j+kM,\theta+p)=e^{-2\pi i k\theta}z_Mg_l(j,\theta)$, then for all $j,k,p\in \mathbb{Z}$, $\displaystyle{\sum_{l\in \mathbb{N}_L}\vert z_Mg_l(j+kM,\theta+p)\vert^2} =\displaystyle{\sum_{l\in \mathbb{N}_L}\vert z_Mg_l(j,\theta)\vert^2}$.\\
Denote $A_j:=\{\theta \in \mathbb{R}:\, \text{ the inequality in (4.2) does not hold}\,\}$ and $A^q_j:=\{\theta \in [0,1[:\, \text{ the inequality in (4.2) does not hold}\,\}$. We show, easily, that for all $j\in \mathbb{N}_M$, $A_j=\displaystyle{\bigcup_{p\in \mathbb{Z}}A^q_j+p}$. For $j\in \mathbb{Z}$, let $j_0\in \mathbb{N}_M$ and $k\in \mathbb{Z}$ such that $j=j_0+kM$, then $A_j=\displaystyle{\bigcup_{p\in \mathbb{Z}}A^q_{j_0}+p}$. Since Lebesgue measure is invariant by translations and that a countable union of measure zero sets is a measure zero set, then we deduce that: $$ \forall j\in \mathbb{Z},\;\displaystyle{\frac{LA}{M}}\leq \displaystyle{\sum_{l\in \mathbb{N}_L}\vert z_Mg_l(j,.)\vert^2}\leq \displaystyle{\frac{LB}{M}}\; \; \text{ a.e. on }\mathbb{R}.$$
\item In this case $A=B=1$, then $(1)$ completes the proof.\\
	\end{enumerate}
	\end{proof}

   The following result is a necessary condition for M-D-G systems to be Riesz bases and to be orthonormal bases when the $z_Mg_l(j,.)$ are continuous.
	\begin{proposition}\label{necessarycont}
Let $N=LM$ and let $g=\{g_l\}_{l\in \mathbb{N}_L }\subset \ell^2(\mathbb{Z})$ such that for all $j\in \mathbb{N}_M$, $l\in \mathbb{N}_L$ $z_Mg_l(j,.)$ is continuous on $\mathbb{R}$. 
	\begin{enumerate}
	\item If $\mathcal{G}(g,L,M,N)$ is a M-D-G Riesz basis for $\ell^2(\mathbb{Z})$, then:\\ 
	$\displaystyle{\bigcap_{l\in \mathbb{N}_L}Z(g_l)}=\emptyset$, where $Z(g_l):=\{(j,\theta)\in \mathbb{Z}\times \mathbb{R}:\; z_Mg_l(j,\theta)=0\}.$
	
	\item If $\mathcal{G}(g,L,M,N)$ is a M-D-G orthonormal basis  for $\ell^2(\mathbb{Z})$, then for all $j\in \mathbb{Z}$ and $\theta \in \mathbb{R}$,
	$$  \displaystyle{\sum_{l\in \mathbb{N}_L}\vert z_Mg_l(j,\theta)\vert^2}= \displaystyle{\frac{L}{M}}.$$
	\end{enumerate}
	\end{proposition}
	\begin{proof}
	By similar arguments used in the proof of proposition \ref{cont}.\\
	\end{proof}
	The next result is a direct consequence of the above proposition and lemma $\ref{lem7}$.
	\begin{corollary}\label{cor7}
	Let $N=LM$ and let $g=\{g_l\}_{l\in \mathbb{N}_L }\subset \ell^1(\mathbb{Z})$.
	\begin{enumerate}
	\item If $\mathcal{G}(g,L,M,N)$ is a M-D-G Riesz basis for $\ell^2(\mathbb{Z})$, then:\\ 
	$\displaystyle{\bigcap_{l\in \mathbb{N}_L}Z(g_l)}=\emptyset$, where $Z(g_l):=\{(j,\theta)\in \mathbb{Z}\times \mathbb{R}:\; z_Mg_l(j,\theta)=0\}.$
	
	\item If $\mathcal{G}(g,L,M,N)$ is a M-D-G orthonormal basis  for $\ell^2(\mathbb{Z})$, then for all $j\in \mathbb{Z}$ and $\theta \in \mathbb{R}$,
	$$  \displaystyle{\sum_{l\in \mathbb{N}_L}\vert z_Mg_l(j,\theta)\vert^2}= \displaystyle{\frac{L}{M}}.$$
	\end{enumerate}
	\end{corollary}
	A direct consequence of corollary \ref{cor7} and lemma $\ref{lem6}$ is the following result.
\begin{corollary}	
Let $M,N,L\in \mathbb{N}$. Let $g:=\{g_l\}_{l\in \mathbb{N}_L}\subset \ell^1(\mathbb{Z})$.
	\begin{enumerate}
	
\item If there exist $l\neq l'$ such that $g_l$ and $g_{l'}$ are odd, then $\mathcal{G}(g,L,M,N)$ is not a M-D-G Riesz basis for $\ell^2(\mathbb{Z})$.\\
\item Assume that $M$ is even. If there exist $l\neq l'$ such that $g_l$ and $g_{l'}$ are even, then $\mathcal{G}(g,L,M,N)$ is not a M-D-G Riesz basis for $\ell^2(\mathbb{Z})$.\\
	\end{enumerate}	
	\end{corollary}
	\begin{remark}
The converse of  $(1)$ in Theorem \ref{necessary} and Proposition \ref{necessarycont} holds if and only if  $L=1$. In fact, if  $L=1$, then $M=N$, therefore the converse is already shown in remark \ref{caseL1}. Assume, now, that $L>1$.  To illustrate that the converse does not hold, in general, here is an example: \\
	Let $L>1$. Let  $g_0=g_1=...=g_{L-1}:=\{e^{-k^2}\}_{k\in \mathbb{Z}}$. Then $g:=\{g_l\}_{l\in \mathbb{N}_L}\subset \ell^1(\mathbb{Z})$ and by \cite{4} $Z(g_l):=\{(j,\theta)\in \mathbb{Z}\times \mathbb{R}:\, z_Mg_l(j,\theta)=0\}=\emptyset$. Then  
$\displaystyle{\bigcap_{l\in \mathbb{N}_L} Z(g_l)}=\emptyset$. But $\mathcal{G}(\,g,L,M,N)$ is not a M-D-G Riesz basis since it contains repetitions; More precisely, each element is repeated $L$ times.\\
	\end{remark}
	
	\begin{remark}
	The converse of  $(2)$ in Theorem \ref{necessary} and Proposition \ref{necessarycont} holds if and only if $L=1$. In fact, if  $L=1$, then $M=N$, therefore the converse is already shown in remark \ref{caseL1}. Assume, now, that $L>1$. Define $g_0\in \ell^2(\mathbb{Z})$ such that for all $j\in \mathbb{N}_M$, $z_Mg_0(j,\theta)=\displaystyle{\frac{1}{\sqrt{M}}}$ on $[0,1[$ ($g_0$ exists since $z_M$ is a unitary from $\ell^2(\mathbb{Z})$ to $\ell^2(q)$). Let $g_0=g_1=...=g_{L-1}$ and take $g:=\{g_l\}_{l\in \mathbb{N}_L}$. Then for all $j\in \mathbb{N}_M$, $\theta \in [0,1[$,  we have $\displaystyle{\sum_{l\in \mathbb{N}_L}\vert z_Mg_l(j,\theta)\vert^2}=\displaystyle{\frac{L}{M}}$. And then, by quasi-periodicity of $z_Mg_l$, we have for all $j\in \mathbb{Z}$, $\theta \in \mathbb{R}$,  we have $\displaystyle{\sum_{l\in \mathbb{N}_L}\vert z_Mg_l(j,\theta)\vert^2}=\displaystyle{\frac{L}{M}}$. But $\mathcal{G}(g,L,M,N)$ is not a M-D-G orthonormal basis for $\ell^2(\mathbb{Z})$ since it contains repetitions; More precisely, each element is repeated $L$ times.\\
	\end{remark}
	\section{K-M-D-G frames}
	In this section we study $K$-M-D-G frames where $K\in B(\ell^2(\mathbb{S}))$. We first give the interesting following result in general K-frame theory.
	\begin{theorem}\label{teta}
Denote by $\mathcal{B}$ the space of all Bessel sequences in $H$ ibdexed by a countable set $\mathcal{I}$. Then the map: 
$$\begin{array}{rcl}
\theta:\mathcal{B}&\longrightarrow& B(H,\ell^2(\mathcal{I})\,)\\
\{f_i\}_{i\in \mathcal{I}}&\mapsto& \theta_f.
\end{array}	$$
where $\theta_f$ is the transform operator of $\{f_i\}_{i\in \mathcal{I}}$, is linear and bijective. Moreover,
$$\begin{array}{rcl}
\theta^{-1}:B(H,\ell^2(\mathcal{I})\,)&\longrightarrow&\mathcal{B}\\
L&\mapsto& \{L^*(e_i)\}_{i\in \mathcal{I}}.
\end{array}$$
where $\{e_i\}_{i\in \mathcal{I}}$ is the standard orthonormal basis for $\ell^2(\mathcal{I})$.\\
	\end{theorem}
	
	\begin{proof}
	It is clear that $\theta$ is linear. Let, now, $f=\{f_i\}_{i\in \mathcal{I}}\in \mathcal{B}$. We have: $$\begin{array}{rcl}
	\theta_f=0&\Longleftrightarrow& \;\forall x\in H,\; \theta_fx=0\\
	&\Longleftrightarrow&\; \forall x\in H,\; \{\langle x, f_i\rangle \}_{i\in \mathcal{I}}=0\\
	&\Longleftrightarrow&\; \forall x\in H,\, \forall i\in \mathcal{I},\; \langle x,f_i\rangle=0\\
	&\Longleftrightarrow&\; \forall i\in \mathcal{I},\; f_i=0\\
	&\Longleftrightarrow&\; f=0.
	\end{array}$$
	where $0$ denoted, without any confusion, the zero of the spaces $\mathbb{K}$, $H
	$, $\ell^2(\mathcal{I})$ and $B(H, \ell^2(\mathcal{I})\,)$.  Then $U$ is injective. Let, now, $L\in B(H,\ell^2(\mathcal{I}))$ and $\{e_i\}_{i\in \mathcal{I}}$ be the standard orthonormal basis for $\ell^2(\mathcal{I})$. We take $f=\{L^*(e_i)\}_{i\in \mathcal{I}}$, then $f$ is well a Bessel sequence for $H$. Moreover, its transform operator is exactely $L$. In fact, for all $x\in H$, we have:
	$$\begin{array}{rcl}
	\theta_f(x)&=& \{\langle x, L^*(e_i)\rangle \}_{i\in \mathcal{I}}\\
	&=&  \{\langle L\,x, e_i\rangle \}_{i\in \mathcal{I}}\\
	&=& L\,x.
	\end{array}$$
	Then the transform operator of $f=\{L^*(e_i)\}_{i\in \mathcal{I}}$ is the operator $L$. i.e. $\theta(f)=L$. Hence $\theta$ is surjective.\\
	\end{proof}
	
	The following result characterizes the operators $K\in B(H)$ for which a given Bessel sequence is a $K$-frame and characterize $K$-dual frames by operators. 
	\begin{proposition}\label{kframe1}
	Let $\{x_i\}_{i\in \mathcal{I}}$ be a Bessel sequence for $H$ whose the synthesis operator is $U_x$. Then  $\{x_i\}_{i\in \mathcal{I}}$ is a $U_xL$-frame for $H$ for all $L\in B(H,\ell^2(\mathcal{I})\,)$. And, conversely, if $\{x_i\}_{i\in \mathcal{I}}$ is a $K$-frame for $H$, where $K\in B(H)$, then there exists $L\in B(H,\ell^2(\mathcal{I})\,)$ such that $K=U_xL$.\\ \\
	Moreover, $\{y_i\}_{i\in \mathcal{I}}\subset H$ is a $K$-dual frame to $\{x_i\}_{i\in \mathcal{I}}$ if and only if there exists $L\in B(H,\, \ell^2(\mathcal{I})\,)$ such that $K=U_xL$ and $g_i=\theta^{-1}L:=L^*(e_i)$ for all $i\in \mathcal{I}$, where $\{e_i\}_{i\in \mathcal{I}}$ is the standard orthonormal basis for $\ell^2(\mathcal{I})$. In particular, we have:
	 $$card\left\{\{y_i\}_{i\in \mathcal{I}} \text{ Bessel sequence K-dual to }\{x_i\}_{i\in \mathcal{I}}\right\}=card\left\{L\in B(H,\ell^2(\mathcal{I})\,): \; K=U_xL\;\right\}.$$
	\end{proposition}
	\begin{proof}
	By Lemma \ref{Gavr2} and Theorem \ref{teta}.
	\end{proof}
	
	\begin{corollary}
	Let $K\in B(H)$. If $\{x_i\}_{i\in \mathcal{I}}$ is $K$-frame for $H$, then  $\{x_i\}_{i\in \mathcal{I}}$ is $KM$-frame for $H$ for all $M\in B(H)$. Moreover, if $\{a_i\}_{i\in \mathcal{I}}$ is a $K$-dual frame to  $\{x_i\}_{i\in \mathcal{I}}$, then $\{M^*(a_i)\}_{i\in \mathcal{I}}$ is a $KM$-dual frame to  $\{x_i\}_{i\in \mathcal{I}}$.\\
		\end{corollary}
\begin{proof}
Denote by $U_x$ the synthesis operator of $\{x_i\}_{i\in \mathcal{I}}$. We have $R(KM)\subset R(K)\subset R(U_x)$. Then $\{x_i\}_{i\in \mathcal{I}}$ is a $KM$-frame for $H$. On the other hand, we have $K=U_x\theta_a$, where $\theta_a$ is the transforme operator of $\{a_i\}_{i\in \mathcal{I}}$. Then $KM=U_x\theta_aM$. Hence a $KM$-dual frame of $\{x_i\}_{i\in \mathcal{I}}$ is $\{(\theta_aM)^*(e_i)\}_{i\in \mathcal{I}}=\{M(\theta_a^*(e_i)\,)\}_{i\in \mathcal{I}}=\{M^*(a_i)\}_{i\in \mathcal{I}}$, where $\{e_i\}_{i\in \mathcal{I}}$ is the standard orthonormal basis for $\ell^2(\mathcal{I})$.\\
\end{proof}

	 By definition of $K$-frames, a M-D-G system $\mathcal{G}(g,L,M,N)$ is a $K$-frame for $\ell^2(\mathbb{S})$ if there exist $0< A\leq B< \infty$ such that for all $f\in \ell^2(\mathbb{S})$, we have:
	$$A\|K^*f\|^2\leq \displaystyle{\sum_{l\in \mathbb{N}_L}\sum_{n\in \mathbb{Z}}\sum_{m\in \mathbb{N}_M}\left \vert \langle f,E_{\frac{m}{M}}T_{nN}g_l\rangle\right \vert^2}\leq B\|f\|^2.$$\\
If $A\|K^*f\|^2= \displaystyle{\sum_{l\in \mathbb{N}_L}\sum_{n\in \mathbb{Z}}\sum_{m\in \mathbb{N}_M}\left \vert \langle f,E_{\frac{m}{M}}T_{nN}g_l\rangle\right \vert^2}$ for all $f\in \ell^2(\mathbb{S})$, the, we say that $\mathcal{G}(g,L,M,N)$ is a tight $K$-M-D-G frame. \\

The following result is a sufficient matrix-condition for a M-D-G system  in $\ell^2(\mathbb{S})$  to be frame.
\begin{theorem}
	Let $g\in \ell^2(\mathbb{S})$ such that $\vert supp(g)\vert < M$ and let $K\in B(\, \ell^2(\mathbb{S})\,)$. Assume that there exist $0<A\leq B< \infty$ such that for all $j\in \mathbb{S}_N$, we have: 
	$$\displaystyle{\frac{A\|K\|^2}{M}}\leq \displaystyle{\sum_{l\in \mathbb{N}_L}\left(\mathcal{M}_{g_l}(j)\mathcal{M}_{g_l}^*(j)\right)_{0,0}}\leq \displaystyle{\frac{B}{M}}.$$
	Then $\mathcal{G}(g,L,N,M)$ is a $K$-M-D-G frame for $\ell^2(\mathbb{S})$ with bounds $A$ and $B$.
	
	\end{theorem}
	
	\begin{proof}
Let $f\in \ell_0(\mathbb{S})$. Since $\vert supp(g_l)\vert< M$ for all $l\in \mathbb{N}_L$, then by lemma \ref{lem4}, we have: $$\displaystyle{\sum_{l\in \mathbb{N}_L}\sum_{n\in \mathbb{Z}}\sum_{m\in \mathbb{N}_M}\vert \langle f,E_{\frac{m}{M}}T_{nN}g_l \rangle \vert^2}=F_1(f)=M\,\displaystyle{\sum_{l\in \mathbb{N}_L}\sum_{j\in \mathbb{S}} \left(\mathcal{M}_{g_l}(j)\mathcal{M}_{g_l}^*(j)\right)_{0,0} \vert f(j)\vert^2}.$$
By lemma \ref{lem3'}, we have for all $j\in \mathbb{S}$, $$\displaystyle{\frac{A\|K\|^2}{M}}\leq \displaystyle{\sum_{l\in \mathbb{N}_L}\left(\mathcal{M}_{g_l}(j)\mathcal{M}_{g_l}^*(j)\right)_{0,0}}\leq \displaystyle{\frac{B}{M}}.$$
Then $$A\|K\|^2\|f\|^2\leq \displaystyle{\sum_{l\in \mathbb{N}_L}\sum_{n\in \mathbb{Z}}\sum_{m\in \mathbb{N}_M}\vert \langle f,E_{\frac{m}{M}}T_{nN}g_l \rangle \vert^2} \leq B\|f\|^2.$$
Thus for all $f\in \ell_0(\mathbb{S})$, we have: $$A\|K^*(f)\|^2\leq \displaystyle{\sum_{l\in \mathbb{N}_L}\sum_{n\in \mathbb{Z}}\sum_{m\in \mathbb{N}_M}\vert \langle f,E_{\frac{m}{M}}T_{nN}g_l \rangle \vert^2} \leq B\|f\|^2.$$
Hence, by density of $\ell_0(\mathbb{S})$ in $\ell^2(\mathbb{S})$, $\mathcal{G}(g,L,M,N)$ is a M-D-G $K$-frame for $\ell^2(\mathbb{S})$.\\
	\end{proof}
The following result is a direct consequence of the proposition \ref{kframe1} for M-D-G systems. It is about construction of M-D-G $K$-frames for $\ell^2(\mathbb{S})$.
\begin{proposition}\label{prop54}
Let $\mathcal{G}(g,L,M,N)$ be a M-D-G Bessel sequence with synthesis operator $U_g$. Then $\mathcal{G}(g,L,M,N)$ is a M-D-G $U_gL$-frame for $\ell^2(\mathbb{S})$ for all\\ $L\in B(\,\ell^2(\mathbb{S}),\;\ell^2(\mathbb{N}_M\times \mathbb{Z}\times \mathbb{N}_L)\;)$.\\
\end{proposition}
\begin{remark}
Let $\{e_{m,n,l}\}_{l\in \mathbb{N}_L,\, n\in \mathbb{Z},\, m\in \mathbb{N}_M}$ be the standard orthonormal basis for $\ell^2(\mathbb{N}_M\times \mathbb{Z}\times \mathbb{N}_L)$. As seen before, $\{L^*(e_{m,n,l})\}_{l\in \mathbb{N}_L,\, n\in \mathbb{Z},\, m\in \mathbb{N}_M}$ is a $U_gL$-dual frame to $\mathcal{G}(g,L,M,N)$. But it is not necessary a M-D-G system.\\
\end{remark}

\begin{proposition}\label{prop56}
 Let $\mathcal{G}(g,L,M,N)$ and $\mathcal{G}(h,L,M,N)$ be two M-D-G Bessel sequences in $\ell^2(\mathbb{S})$ such that $U_g$ and $\theta_h$ are the synthesis operator of $\mathcal{G}(g,L,M,N)$ and the transform operator of $\mathcal{G}(h,L,M,N)$ respectively. Denote $S_{h,g}:=U_g\theta_h$. Then $\mathcal{G}(g,L,M,N)$ is a $S_{h,g}$-M-D-G frame for $\ell^2(\mathbb{S})$. Moreover, $\mathcal{G}(h,L,M,N)$ is a $S_{h,g}$-dual frame to  $\mathcal{G}(g,L,M,N)$.\\
\end{proposition}
\begin{proof}
We have $S_{h,g}=U_g\theta_h$, then by proposition \ref{kframe1}, $\mathcal{G}(g,L,M,N)$ is a $S_{h,g}$-frame for $\ell^2(\mathbb{S})$ and $\theta^{-1}(\theta_h)=\mathcal{G}(h,L,M,N)$ is a $S_{h,g}$-dual frame. \\
\end{proof}
	
\begin{corollary}
Let $L,M,N\in \mathbb{N}$ be arbitrary. Let $g=\{g_l\}_{l\in \mathbb{N}_L},h=\{h_l\}_{l\in \mathbb{N}_L}\in \ell_0(\mathbb{S})$. Then $\mathcal{G}(g,L,M,N)$ is a $S_{h,g}$-M-G-G frame for $\ell^2(\mathbb{S})$, where $S_{h,g}:=U_gU_h^*$.\\
\end{corollary}	
\begin{proof}
Since the M-D-G system associated to a sequence in $\ell_0(\mathbb{S})$ is a Bessel sequence in $\ell^2(\mathbb{S})$, proposition \ref{prop56} completes the proof.\\
\end{proof}
\begin{remark}
Unlike M-D-G ordinary frames, there is no conditions on the parameters $M$, $N$ and $L$ for M-D-G $K$-frames. Propositions \ref{prop54} and \ref{prop56} give construction methods to construct $K$-M-D-G frames which are not M-D-G frames for $\ell^2(\mathbb{S})$.\\
\end{remark}
\begin{example}
Let $M,N,L\in \mathbb{N}$ be such that $N>LM$ and let $g=\{g_l\}_{l\in \mathbb{N}_L}\subset \ell_0(\mathbb{Z})$ with synthesis operator $U_g$. Let $L\in B(\,\ell^2(\mathbb{S}),\;\ell^2(\mathbb{N}_M\times \mathbb{Z}\times \mathbb{N}_L)\;)$ be arbitrary. Then $\mathcal{G}(g,L,M,N)$ is a M-D-G $U_gL$-frame for $\ell^2(\mathbb{Z})$ but not a M-D-G frame  since $N>LM$.\\
\end{example}

Let $\mathcal{G}(u,L,M,N)$ and $\mathcal{G}(v,L,M,N)$ be two  M-D-G Bessel systems for $\ell^2(\mathbb{S})$ whose the synthesis operators are $U_u$ and $U_v$ repectively. Denote $S_{u,v}=U_uU_v^*$. In what above, we have shown that $\mathcal{G}(u,L,M,N)$ is a M-D-G $S_{u,v}$-frame and $\mathcal{G}(v,L,M,N)$ is a $S_{u,v}-$dual M-D-G frame. The following result gives other M-D-G systems which are  M-D-G $S_{u,v}$-frames and have $S_{u,v}$-dual M-D-G frames.
\begin{proposition}
Let $\mathcal{G}(g,L,M,N)$ be a M-D-G frame for $\ell^2(\mathbb{S})$ with frame operator $S$. Then:
\begin{enumerate}
\item $\mathcal{G}(g,L,M,N)$ is a M-D-G $S_{u,v}$-frame for $\ell^2(\mathbb{S})$ and $\mathcal{G}(S_{v,u}S^{-1}g,L,M,N) $ is a $S_{u,v}$-dual frame. 
\item $\mathcal{G}(S_{u,v}g,L,M,N)$ is a M-D-G $S_{u,v}$-frame for $\ell^2(\mathbb{S})$ and $\mathcal{G}(S^{-1}g,L,M,N) $ is a $S_{u,v}$-dual frame.
 \end{enumerate}
Where for an operator $M\in B(\, \ell^2(\mathbb{S})\,)$, $Mg:=\{Mg_l\}_{l\in \mathbb{N}_L}$.\\
\end{proposition}
\begin{proof}
Since $\mathcal{G}(g,L,M,N)$ is a M-D-G frame for $\ell^2(\mathbb{S})$ with frame operator $S$, then, by lemma \ref{lem2}, for all $f\in \ell^2(\mathbb{S})$, we have $f=\displaystyle{\sum_{l\in \mathbb{N}_L}\sum_{n\in \mathbb{Z}}\sum_{m\in \mathbb{N}_M}\langle f, E_{\frac{m}{M}}T_{nN}S^{-1}g_l\rangle E_{\frac{m}{M}}T_{nN}g_l}$.\\
\begin{enumerate}
\item Then for all $f\in \ell^2(\mathbb{S})$, we have:
$$\begin{array}{rcl}
S_{u,v}f&=&\displaystyle{\sum_{l\in \mathbb{N}_L}\sum_{n\in \mathbb{Z}}\sum_{m\in \mathbb{N}_M}\langle S_{u,v}f, E_{\frac{m}{M}}T_{nN}S^{-1}g_l\rangle E_{\frac{m}{M}}T_{nN}g_l}\\
&=& \displaystyle{\sum_{l\in \mathbb{N}_L}\sum_{n\in \mathbb{Z}}\sum_{m\in \mathbb{N}_M}\langle f, S_{v,u}E_{\frac{m}{M}}T_{nN}S^{-1}g_l\rangle E_{\frac{m}{M}}T_{nN}g_l}\\
&=&  \displaystyle{\sum_{l\in \mathbb{N}_L}\sum_{n\in \mathbb{Z}}\sum_{m\in \mathbb{N}_M}\langle f, E_{\frac{m}{M}}T_{nN}S_{v,u}S^{-1}g_l\rangle E_{\frac{m}{M}}T_{nN}g_l}\;\;\; (\text{ lemma } \ref{lem2}).
\end{array}$$
Hence $\mathcal{G}(g,L,M,N)$ is a M-D-G $S_{u,v}$-frame for $\ell^2(\mathbb{S})$ and $\mathcal{G}(S_{v,u}S^{-1}g,L,M,N) $ is a $S_{u,v}$-dual frame. \\
\item In the other hand, for all $f\in \ell^2(\mathbb{S})$, we have:
$$\begin{array}{rcl}
S_{u,v}f&=&\displaystyle{\sum_{l\in \mathbb{N}_L}\sum_{n\in \mathbb{Z}}\sum_{m\in \mathbb{N}_M}\langle f, E_{\frac{m}{M}}T_{nN}S^{-1}g_l\rangle S_{u,v}E_{\frac{m}{M}}T_{nN}g_l}\\
&=&  \displaystyle{\sum_{l\in \mathbb{N}_L}\sum_{n\in \mathbb{Z}}\sum_{m\in \mathbb{N}_M}\langle f, E_{\frac{m}{M}}T_{nN}S^{-1}g_l\rangle E_{\frac{m}{M}}T_{nN}S_{u,v}g_l}\;\;\; (\text{ lemma }\ref{lem2}).\\
\end{array}$$
Hence $\mathcal{G}(S_{u,v}g,L,M,N)$ is a M-D-G $S_{u,v}$-frame for $\ell^2(\mathbb{S})$ and $\mathcal{G}(S^{-1}g,L,M,N) $ is a $S_{u,v}$-dual frame.
\end{enumerate}
\end{proof}
The following result shows the equivalence between unique  $K$-dual frames and $K$-minimal frames. 
\begin{theorem}
Let $K\in B(H)$ and $\{f_i\}_{i\in \mathcal{I}}$ be a $K$-frame with synthesis operator $U_f$. Then the following statements are equivalent:
\begin{enumerate}
\item $\{f_i\}_{i\in \mathcal{I}}$ is unique $K$-dual.
\item There exists a unique $L\in B(H,\ell^2(\mathcal{I})\,)$ such that $K=U_fL$.
\item $U_f$ is injective.
\item $\{f_i\}_{i\in \mathcal{I}}$ is $K$-minimal.
\end{enumerate}
\end{theorem}
\begin{proof}\hspace{1cm}
\begin{enumerate}
\item[]$(1)\Longleftrightarrow(2):$ It is a direct consequence of proposition \ref{kframe1}.
\item[]$(2)\Longrightarrow (3):$ Assume $(2)$, then $\{f_i\}_{i\in \mathcal{I}}$ has a unique $K$-dual $\{L^*(e_i)\}_{i\in \mathcal{I}}$. We can suppose that $L^*(e_i)\neq 0$ for all $i\in \mathcal{I}$. Suppose, by contradiction, that $U_f$ is not injective, then there exists $x_0\neq 0$ such that $U_f(x_0)=0$, i.e. $\displaystyle{\sum_{i\in \mathcal{I}}\langle x_0, e_i\rangle f_i}=0$. Let $j\in \mathcal{I}$ such that $\langle x_0, e_j\rangle\neq 0$. Then $f_j=\displaystyle{\sum_{i\neq j}\alpha_i\,f_i}$ where $\alpha_i=\displaystyle{\frac{\langle x_0, e_i\rangle }{\langle x_0, e_j\rangle}}$ for all $i\neq j$. It is clear that $\{\alpha_i\}_{i\in \mathcal{I}-\{j\}}\in \ell^2(\mathcal{I}-\{j\})$.

We have  for all $x\in H$, $$\begin{array}{rcl}

 Kx&=&\displaystyle{\sum_{i\in \mathcal{I}}\langle x,L^*(e_i)\rangle f_i}\\
 &=& \displaystyle{\sum_{i\neq j}\langle x,L^*(e_i)\rangle f_i}+\langle x, L^*(e_j)\rangle f_j\\
 &=& \displaystyle{\sum_{i\neq j}\langle x,L^*(e_i)\rangle f_i}+\langle x, L^*(e_j)\rangle \displaystyle{\sum_{i\neq j}\alpha_if_i}\\
 &=&\displaystyle{\sum_{i\neq j}\langle x, L^*(e_i+\alpha_i\,e_j)\rangle f_i}.\\
\end{array}$$
Since $\{\alpha_i\}_{i\in \mathcal{I}-\{j\}}\in \ell^2(\mathcal{I}-\{j\})$, we show, easily, that $\{g_i\}_{i\in \mathcal{I}}$ is  a Bessel sequence for $H$, where $g_i=L^*(e_i+\alpha_i\,e_j)$ for $i\neq j$ and $g_j=0$.  Then $\{g_i\}_{i\in \mathcal{I}}$ is another $K$-dual frame to $\{f_i\}_{i\in \mathcal{I}}$ than $\{L^*(e_i)\}_{i\in \mathcal{I}}$ since $L^*(e_j)\neq 0=g_j$. Contradiction.
\item[]$(3)\Longrightarrow (2):$ If $U_f$ is injective, then $L$ is, clearly, unique.
\item[]$(3)\Longleftrightarrow(4)$: By definition of $K$-minimal frames.
\end{enumerate}
\end{proof}

\begin{remark}
The $K$-minimality ( unicity of $K$-dual frame) of a $K$-frame does not depend on the operator $K$. i.e. Let $\{x_i\}_{i\in \mathcal{I}}$ be a $K$-frame and also a $M$-frame for $H$, where $K,M\in B(H)$. If $\{x_i\}_{i\in \mathcal{I}}$  is $K$-minimal, then it is also $M$-minimal.
\end{remark}

The following example is an example of $K$-minimal (unique $K$-dual) M-D-G frames.
\begin{example}
Let $L\geqslant 2$ be an integer and $M,N\in \mathbb{N}$ such that $N=LM$. Let $\mathcal{G}(g,L,M,N)$ be a M-D-G orthonormal basis for $\ell^2(\mathbb{Z})$ (the existence is proved in Theorem \ref{exortho}). Let $K\in B(\ell^2(\mathbb{Z})\,)$ defined as follows: For all $m\in \mathbb{N}_M, n\in \mathbb{Z}$ and $l\in \mathbb{N}_{L-1}$, $K(g_{m,n,l})=g_{m,n,l}$ and $K(g_{m,n,L-1})= 0.$ 
We have for all $f\in \ell^2(\mathbb{Z})$, 
$$\begin{array}{rcl}
\|K^*(f)\|^2&=&\displaystyle{\sum_{l\in \mathbb{N}_L}\sum_{n\in \mathbb{Z}}\sum_{m\in \mathbb{N}_M}\vert\langle K^*(f),g_{m,n,l}\rangle\vert^2}\\
&=& \displaystyle{\sum_{l\in \mathbb{N}_L}\sum_{n\in \mathbb{Z}}\sum_{m\in \mathbb{N}_M}\vert\langle f,K(g_{m,n,l})\rangle\vert^2}.
\end{array}$$
Hence $\mathcal{G}(\, \{g_l\}_{l\in \mathbb{N}_{L-1}}, M,N)$ is a  M-D-G  Parseval $K$-frame. Moreover its synthesis operator is injective. Then $\mathcal{G}(\, \{g_l\}_{l\in \mathbb{N}_{L-1}}, M,N)$ is a $K$-minimal (unique $K$-dual) M-D-G $K$-frame for $\ell^2(\mathbb{Z})$.
\end{example}
\begin{remark}
We have seen in Theorem \ref{Rieszbasis} that for M-D-G ordinary frames, a M-D-G frame $\mathcal{G}(g,L,M,N)$ for $\ell^2(\mathbb{S})$ is minimal (Riesz basis/ exact) if and only if $card(\mathbb{S}_N)=LM$. Observe in the above example that this condition does not hold for M-D-G $K$-frames. In fact, 
since  $\mathcal{G}(g,L,M,N)$ is a M-D-G orthonormal basis for $\ell^2(\mathbb{Z})$, then $N=LM$. Thus,  $N>(L-1)M$ and $\mathcal{G}(\, \{g_l\}_{l\in \mathbb{N}_{L-1}}, M,N)$ is a $K$-minimal (unique $K$-dual) M-D-G $K$-frame for $\ell^2(\mathbb{Z})$.
\end{remark}
	\medskip
\section*{Acknowledgments}
	It is my great pleasure to thank the referee for his careful reading of the paper and for several helpful suggestions.	
	
	\section*{Declarations}
	
	\medskip
	
	\noindent \textbf{Availablity of data and materials}\newline
	\noindent Not applicable.
	
	\medskip
	
	\noindent \textbf{Competing  interest}\newline
	\noindent The author declares that he has no competing interests.
	
	\medskip
	
	\noindent \textbf{Fundings}\newline
	\noindent  The author declare sthat there is no funding available for this article.

\end{document}